\documentclass[11pt,reqno]{amsart}

\usepackage{amsmath}
\usepackage{amsthm}
\usepackage{amssymb}
\usepackage{enumerate}
\usepackage{amscd}
\usepackage{psfrag}

\usepackage[latin1]{inputenc}

\usepackage[pdftex]{graphicx,color}
\title[]{Deforming  three-manifolds with positive scalar curvature}

\author{Fernando Coda Marques}

\date{}

\newtheorem{thm}{Theorem}[section]

\newtheorem*{thm*}{Theorem}

\newtheorem*{mainthm}{Main Theorem}

\newtheorem{prop}[thm]{Proposition}

\newtheorem{lem}[thm]{Lemma}

\newtheorem{cor}[thm]{Corollary}

\numberwithin{equation}{section}

\numberwithin{figure}{section}

\def\g{\rightarrow}

\def\d{\partial}

\def\re{\mathbb{R}}

\def\XXint#1#2#3{{\setbox0=\hbox{$#1{#2#3}{\int}$}
     \vcenter{\hbox{$#2#3$}}\kern-.5\wd0}}

\begin{document}


\begin{abstract}
In this paper we prove that the moduli space  of metrics with positive scalar curvature of an orientable 
compact 3-manifold   is path-connected. The proof uses the Ricci flow with surgery, the conformal method,  and the connected sum construction of Gromov and Lawson. The work of Perelman on Hamilton's Ricci flow is fundamental. As one of the applications we  prove the path-connectedness of the space of trace-free asymptotically flat solutions to the vacuum Einstein constraint equations
on $\re^3$.
\end{abstract}

\maketitle




\section{Introduction}

\footnotetext[1]{2000 {\it Mathematics Subject Classification.} Primary 53C21; Secondary 53C80, 83C05.}

In 1916 H. Weyl  proved the following result:
\begin{thm*}[\cite{WEYL16}]\label{weyl}
Let $g$ be a metric of positive scalar curvature on the two-sphere $S^2$. There exists a continuous
path of metrics $\mu \in [0,1] \g g_\mu$ on $S^2$, of positive scalar curvature, such that $g_0=g$ and $g_1$ has constant curvature.
\end{thm*}
The interest in such deformations came from  the idea of using the continuity method to find an isometric embedding of  $(S^2,g)$ as a convex surface in
$\re^3$ (compare \cite{NIRENBERG53}). His  proof is an application of the celebrated  Riemann's Uniformization Theorem. It follows from uniformization that there exists a constant curvature metric $\overline{g}$ in the conformal class of $g$. If $\overline{g} = e^{2f}g$, then it is easy to check that
$g_\mu = e^{2\mu f}g$ has positive scalar curvature for every $\mu \in [0,1]$. The space of metrics of positive scalar curvature on $S^2$ is in fact contractible, as verified by J. Rosenberg and S. Stolz in \cite{ROSENBERGSTOLZ01}.  

 It is then natural to  look for analogues of the above result in higher dimensions. Of course  there is no uniformization theorem 
 available in general, hence other tools have to be introduced.  We will  explain in Section  \ref{conformal} of this paper
 that  H. Weyl's argument  extends to dimensions greater than two, provided the metrics are in the same conformal class.  

The object of this paper will be to prove that the moduli space  of metrics with positive scalar curvature of an orientable 
compact 3-manifold   is path-connected. If $M$ is a compact manifold, we will denote by $\mathcal{R}_+(M)$ the set of Riemannian metrics $g$ on $M$ with positive scalar curvature $R_g$. The associated moduli space is the quotient $\mathcal{R}_+(M)/{\rm Diff}(M)$ of $\mathcal{R}_+(M)$ under the standard  action of the group of diffeomorphisms ${\rm Diff}(M)$.  We refer the reader to \cite{ROSENBERG07} for a nice  survey on recent results about the space of metrics of positive scalar curvature on a given smooth manifold. Unless otherwise specified, the space of metrics on a given manifold will be endowed with the $C^\infty$ topology.

It will be convenient to call  the positive scalar curvature metrics $g$ and $g'$ {\it isotopic} to each other
if there exists a continuous path $\mu \in [0,1] \g g_\mu \in \mathcal{R}_+(M)$ such that $g_0=g$ and $g_1=g'$, i.e., if $g$ and $g'$ lie in the same path-connected component of $\mathcal{R}_+(M)$.

Our main theorem is:
\begin{mainthm}\label{general.thm}
Suppose that $M^3$ is a compact orientable 3-manifold such that $\mathcal{R}_+(M) \neq \emptyset$. 
 Then the moduli space $\mathcal{R}_+(M) / {\rm Diff}(M)$ is path-connected. 
\end{mainthm}

In \cite{CERF68}, J. Cerf proved that the set ${\rm Diff}_+(S^3)$ of orientation-preserving diffeomorphisms of the 3-sphere is path-connected.  For this reason the statement for $S^3$ is stronger:
\begin{cor}\label{sphere}
 The space $\mathcal{R}_+(S^3)$ of positive scalar curvature metrics on the 3-sphere is path-connected.
\end{cor}

{\it Remark:} We refer the reader to  \cite{SMALE59} and  \cite{HATCHER83} for results on the homotopy type of ${\rm Diff}_+(S^2)$ and 
${\rm Diff}_+(S^3)$, respectively. The path-connectedness of ${\rm Diff}_+(S^2)$ was also proved in \cite{MUNKRES60}.

The picture in higher dimensions is quite different. This was first noticed by N. Hitchin (\cite{HITCHIN74}) in 1974, where he  proves that the spaces $\mathcal{R}_+(S^{8k})$ and $\mathcal{R}_+(S^{8k+1})$ are disconnected for each $k \geq 1$. This result follows from the consideration
of index-theoretic invariants associated to the Dirac operator of spin geometry. It holds in general for all the compact spin manifolds $X$ of
dimensions $8k$ and $8k+1$ with $\mathcal{R}_+(X) \neq \emptyset$. In 1988 R. Carr (\cite{CARR88}) proved that the space $\mathcal{R}_+(S^{4k-1})$ has infinitely many connected components for each $k \geq 2$. In dimension 7 ($k=2$ case) 
this result was 
proved earlier by Gromov and Lawson in 1983 (see Theorem 4.47 of \cite{GROMOVLAWSON83}). It was improved by
M. Kreck and S. Stolz (\cite{KRECKSTOLZ93}) in 1993, where they show  that even the moduli space $\mathcal{R}_+(S^{4k-1})/{\rm Diff}(S^{4k-1})$ has infinitely many
connected components for $k \geq 2$. The same statement holds true for any nontrivial spherical quotient of dimension greater than or equal to five, as proved
by  B. Botvinnik and P. Gilkey in \cite{BOTVINNIKGILKEY96}.   The surgery arguments used in these proofs break down in the three-dimensional case.

In his famous 1982 paper R. Hamilton (\cite{HAMILTON82}) introduced the equation
$$
\frac{\d g}{\d t} = -2 Ric_g,
$$
known as the Ricci flow, and proved the existence of short time solutions with arbitrary compact Riemannian manifolds as initial conditions. He also proved that if $g(t)$ denotes a solution to the Ricci flow on a compact 3-manifold $M$ such that $g(0)$ has positive Ricci curvature, then the flow becomes extinct at finite time $T>0$, $Ric_{g(t)}>0$ for all $t \in [0,T)$, and  the volume one rescalings $\tilde{g}(t)$ of $g(t)$ converge to a constant curvature metric as $t \g T$.
 
The evolution equation for the scalar curvature is
\begin{equation}\label{ricci.flow}
\frac{\d R_g}{\d t} = \Delta R_g + 2|Ric_g|^2,
\end{equation}
from which follows by Maximum Principle that the condition of positive scalar curvature is preserved by Ricci flow in any dimension $n$.  In fact, if $R_{g(0)}\geq R_0 >0$, it follows from equation (\ref{ricci.flow}) that
$$
\min_M R_{g(t)} \geq \frac{1}{\frac{1}{R_0}-\frac{2}{n}t},
$$
which forces the flow to end in finite time. These facts  make the Ricci flow a 
natural tool in the study of deformations of metrics with positive scalar curvature.

The great difficulty is that the condition of $R_g>0$ is too weak to imply convergence results. Unlike in the case of positive Ricci curvature, singularities can occur in proper subsets of the manifold. In order to deal with this kind of situation Hamilton introduced in \cite{HAMILTON97}, in the context of four-manifolds, a discontinuous 
evolution process known as Ricci flow with surgery. This is a collection of successive standard Ricci flows, each of them defined in the maximal
interval of existence, such that the initial condition of each flow is obtained from the preceding flow by topological and geometrical operations at
the singular time. These operations constitute what is known as surgery, and  are devised to eliminate the regions of the manifold where singularities develop,  replacing them by regions of standard geometry. 

In two dimensions no surgery is needed. In \cite{HAMILTON88}, Hamilton proved that if $g$ has positive scalar curvature (or Gauss curvature)
on $S^2$, then the solution to the normalized Ricci flow with initial condition $(S^2,g)$ converges to a constant curvature metric.  (See \cite{CHOW91} for an extension to arbitrary $g$). His proof was made independent of uniformization by Chen, Lu and Tian in \cite{CHENLUTIAN06}. This is a 
heat flow proof of Weyl's theorem.

In three dimensions the existence of a Ricci flow with surgery and the study of its properties 
were accomplished  by G. Perelman in a series of three papers \cite{PERELMAN02}, \cite{PERELMAN03A}, \cite{PERELMAN03B}. One of Perelman's breakthroughs was the understanding of how singularities form, which allowed him to restrict the surgeries
to almost cylindrical regions. When the Ricci flow with surgery ends in finite time, it is possible, by reasoning back in time, to recover the original topology of the manifold. In fact  he proves that if the Ricci flow with surgery of an orientable compact Riemannian 3-manifold becomes extinct in finite time, then the manifold
is diffeomorphic to a connected sum of spherical space forms and finitely many copies of $S^2 \times S^1$. Since he is also able to prove 
(see \cite{PERELMAN03B}) that
this is in fact the case if the fundamental group is trivial (or a free product of finite and infinite cyclic groups, more generally), a proof
of the  Poincaré Conjecture is obtained as an  application.  A different argument for the finite extinction time result is 
due to T. Colding and B. Minicozzi (see \cite{COLDINGMINICOZZI05}).

Another application is the topological classification of the orientable compact 3-manifolds which admit metrics of positive scalar curvature (see 
\cite{SCHOENYAU79A} and \cite{GROMOVLAWSON83} for earlier results with different methods). Since the surgeries only increase scalar curvature, the associated Ricci flows with surgery have to become extinct in finite time. We also know that the 
condition of positive scalar curvature is stable under connected sums (see \cite{GROMOVLAWSON80} and \cite{SCHOENYAU79}).  Therefore the assumption of the 
Main Theorem is equivalent to saying that $M$ is diffeomorphic to a connected sum of spherical space forms and finitely many copies of 
$S^2 \times S^1$. 

Our method of proof is going to be a combination of the heat flow technique (Ricci flow with surgery), and the conformal method. The work of Perelman on Hamilton's Ricci flow is fundamental (\cite{PERELMAN02}, \cite{PERELMAN03A}, and  \cite{PERELMAN03B}). We refer the reader to \cite{CAOZHU06}, \cite{KLEINERLOTT08}, and \cite{MORGANTIAN07} for some detailed presentations of the arguments due to Perelman. See also \cite{BESSONetal08}, \cite{BESSONetal09} and \cite{MORGANTIAN08}. In this paper
we choose to  follow more closely the exposition of the book by J. Morgan and G. Tian (see \cite{MORGANTIAN07}).  

Since we are not only interested in the topology, but also in the geometry, we need something to undo surgeries in a certain sense. We will achieve that by means of the connected sum construction of 
Gromov and Lawson of \cite{GROMOVLAWSON80}(see \cite{SCHOENYAU79} for a related construction). Recall that the Gromov-Lawson connected sum construction
is a way of putting a metric of positive scalar curvature on the connected sum $(M_1,g_1) \# (M_2,g_2)$, provided $g_1$ and $g_2$ have positive
scalar curvature themselves. The resulting manifold is a disjoint union of the complements $M_i \setminus B_\delta(p_i)$ of small balls, $i=1,2$,
with their original metrics, and a neck region $N$.

In order to explain our strategy let us introduce the concept of a canonical metric. 
Let $h$ be the metric on the unit sphere $S^3$ induced by the standard inclusion $S^3 \subset \re^4$. A {\it canonical metric}  is any metric
obtained from the 3-sphere $(S^3,h)$ by attaching to it finitely many constant curvature spherical quotients (through the Gromov-Lawson procedure), and adding to 
it finitely many handles (Gromov-Lawson connected sums of $S^3$ to itself). The resulting manifold $M$ is diffeomorphic to
$$
S^3 \# (S^3/\Gamma_1) \# \dots \# (S^3/\Gamma_k) \# (S^2 \times S^1)
\# \dots \# (S^2 \times S^1),
$$
where $\Gamma_1,\dots,\Gamma_k$ are finite subgroups of $SO(4)$ acting freely on $S^3$. The number of $S^2 \times S^1$ summands coincides
with the number of handles attached, and the spherical quotients come with a choice of orientation. The resulting metric $\hat{g}$ is locally conformally flat and has positive scalar curvature. Two canonical 
metrics on $M$ are in the same path-connected component of the moduli space $\mathcal{R}_+(M)/{\rm Diff}(M)$.

\begin{figure}\label{canonicalmetric}
\begin{center}
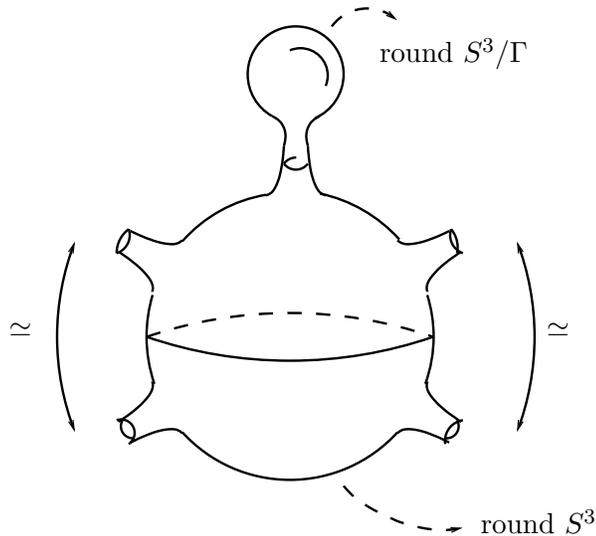
\caption{A canonical metric on $S^3/\Gamma \# (S^2\times S^1)\# (S^2\times S^1)$.}
\end{center}
\end{figure}

Given a metric $g_0$ in $\mathcal{R}_+(M)$, the strategy is to use the Ricci flow with surgery $(M^3_i,g_i(t))_{t \in [t_{i},t_{i+1})}$ with initial condition $g_0(0)=g_0$ to construct a continuous path in $\mathcal{R}_+(M)$ that starts at $g_0$ and ends at a canonical metric.  As in the proof of the Poincar\'{e} Conjecture  we use backwards induction on
the set of singular times $t_i$. We will give an outline of the proof in Section \ref{outline}.

We also give some applications to General Relativity in Section \ref{relativity}. We are interested in studying the topology of spaces 
of asymptotically flat metrics under natural scalar curvature conditions. For simplicity we restrict ourselves to $\re^3$, although the methods
can also be applied to other manifolds (see final remark). In Section \ref{relativity} we will always use
the topology induced by weighted H\"{o}lder norms $C^{k,\alpha}_\beta$.

We will prove that three different spaces are path-connected (Theorems \ref{connectedness.1}, \ref{connectedness.2}, and \ref{connectedness.3}). A metric $g$ on $\re^3$ will be called  {\it asymptotically flat}  if
 $g_{ij}-\delta_{ij} \in C^{2,\alpha}_{-1}$. In particular
 $$
 |g_{ij}-\delta_{ij}|(x) + |x||\d\, g_{ij}|(x) + |x|^2|\d^2 g_{ij}|(x) = O(1/|x|)
 $$
 as $x \g \infty$.
 
Let $\mathcal{M}_1$ be the set of asymptotically flat metrics on $\re^3$ of zero scalar curvature. The first application is:
\begin{thm}\label{connected.1}
The set $\mathcal{M}_{1}$ is path-connected in the $C^{2,\alpha}_{-1}$ topology.
\end{thm}

The idea is to first deform a metric in $\mathcal{M}_1$ into one that can be conformally compactified, i.e., one that can be obtained as a blow-up 
$G_x^4\, g$ of a positive
scalar curvature metric $g$ on $S^3$. Here $G_x$ denotes   the Green's function associated to the conformal Laplacian
$L_g=\Delta_g-\frac18 R_g$ of $g$, with pole at $x \in S^3$. By deforming $g$, the Corollary \ref{sphere} can be used to construct a continuous path of
asymptotically flat and scalar-flat metrics on $\re^3$ connecting $G_x^4\, g$ to the flat metric. 

As a consequence  we prove
\begin{thm}\label{connected.2}
Let $\mathcal{M}_2$ be the set of asymptotically flat metrics $g$ on $\re^3$ such that $R_g \geq 0$, and $R_g \in L^1$. 
Then the  set $\mathcal{M}_{2}$ is path-connected in the $C^{2,\alpha}_{-1}$ topology.
\end{thm}

This question had been studied previously by B. Smith and G. Weinstein through a parabolic method. In \cite{SMITHWEINSTEIN04}, the authors proved path-connectedness
of the space of metrics in $\mathcal{M}_2$ that admit
 a quasi-convex global foliation. Once we  have  established Theorem \ref{connected.1}, Theorem \ref{connected.2} will follow by the conformal method.

The final application concerns trace-free asymptotically flat solutions to the vacuum Einstein constraint equations
on $\re^3$.   The solutions to
the constraint equations parametrize the space of solutions to the vacuum Einstein equations because of the well-posedness
of the initial value formulation, as proved by Y. Choquet-Bruhat. We refer the reader to \cite{BARTNIKISENBERG04} for a nice survey on the constraint equations.

We say that $(g,h)$ is an {\it asymptotically flat initial data set} on $\re^3$ if $g$ is a Riemannian metric on $\re^3$ such that
 $g_{ij}-\delta_{ij} \in C^{2,\alpha}_{-1}$, and $h$ is a symmetric $(0,2)$-tensor with $h_{ij} \in C^{1,\alpha}_{-2}$.
 
 Let $\mathcal{M}_3$ be the set of all asymptotically flat initial data sets $(g,h)$ on $\re^3$ such that
 \begin{itemize}
 \item [a)] $tr_g\, h = 0$,
 \item [b)] $R_g = |h|^2$,
 \item [c)] and $(div_g\, h)_j:=\nabla _i h^{i}_{\ j} = 0$.
 \end{itemize}
 
 The theorem is:
 \begin{thm}\label{connected.3}
 The set $\mathcal{M}_3$ is path-connected in the $C^{2,\alpha}_{-1} \times C^{1,\alpha}_{-2}$-topology.
 \end{thm}
 
 The full set of solutions to the vacuum Einstein constraint equations  is the set $\mathcal{M}_4$ of all asymptotically flat initial data sets $(g,h)$ defined  on $\re^3$ such that
 \begin{itemize}
 \item [a)] $R_g + (tr_g\, h)^2 - |h|^2=0$,
 \item [b)]  $\nabla _i h^{i}_{\ j}-\nabla_j (tr_g\, h) = 0$.
 \end{itemize}
 These metrics  no longer have nonnegative scalar curvature, so it would be  interesting  to find methods to study their deformations.
 
 The paper is organized as follows. In Section 2  we will explain the main steps in the proof of the Main Theorem with a minimum of notation. In Section 3 we explain the conformal method. In Section 4 we prove some interpolation lemmas which will be useful later in handling regions covered by necks. In Section 5 we discuss the surgery process and some of its basic properties. In Section 6 we recall the connected sum construction of manifolds of positive scalar curvature due to Gromov and Lawson, and explain how it can be used to revert surgery. In Section 7 we discuss some of the basic results about
 the Ricci flow with surgery. In Section 8 we introduce the concept of a canonical metric and  give a proof of the main theorem.
In Section 9 we prove the connectedness results concerning asymptotically flat metrics on $\re^3$.

{\it Acknowledgments.} This project started in February of 2008 while I was visiting  Gang Tian at Princeton University. I am 
deeply grateful to  Tian for the support and the many enlightening discussions about the nature of singularity formation during the Ricci flow with surgery. I also thank Andre Neves for the mathematical conversations we had during that time.  I am  thankful to  Richard Schoen for the interest and for suggesting the problem back in 2006.   I would  like to thank Gerard Besson for discussions about the surgery process, and I also thank the anonymous referee for the invaluable suggestions on the exposition. Finally I am grateful to the hospitality of the Institute for Advanced Study, in Princeton, where part of this work was written during  the fall of 2008. I was supported by CNPq-Brazil and FAPERJ.


\section{Outline of the proof of the Main Theorem}\label{outline}

In this section we will  give an outline of  the proof of our main theorem.  For that we will use the terminology  associated with the 
 Ricci flow with surgery (see Section \ref{section.ricci} of this paper).

Let $g_0$ be  a  positive scalar curvature metric on $M^3$. We will consider
$(M^3_i,g_i(t))_{t \in [t_i,t_{i+1})}$ to be  the Ricci flow with surgery  with initial condition $(M^3,g_0)$. 

 The proof is by (backwards) induction on the set of singular times $t_i$. The goal is to prove that every connected component of $M_i$ is isotopic to a canonical metric, for any singular time $t_i$.
 Since the scalar curvature is positive at time $t=0$, and scalar curvature only increases with surgeries, the parabolic maximum principle implies that the flow has  to become extinct
in finite time. This means that there exists $j \geq 0$ such that $M_{j+1}=\emptyset$.

Therefore we should start with the components of $M_j$. Since $M_{j+1}=\emptyset$, the continuing region $C_{t_{j+1}}$ at the extinction time $t_{j+1}$, as defined by Perelman, is empty too. This means that just before the extinction, at time $t'=t_{j+1}-\eta$, for some small $\eta > 0$,  the scalar curvature of
$g_{t'}$ is uniformly and sufficiently large
 so that every point of $(M_j,g_{t'})$ is contained in a canonical neighborhood.

Let us recall that there  are four kinds of canonical  neighborhoods: i) $\varepsilon$-necks; ii) $(C,\varepsilon)$-caps; iii) $C$-components; iv) $\varepsilon$-round components.
An $\varepsilon$-neck centered at $x \in (M^3,g)$ is a submanifold $N\subset M$ and a diffeomorphism $\psi_N:S^2 \times (-1/\varepsilon,1/\varepsilon) \g N$ such that the metric $R_g(x)\psi^*(g)$ is $\varepsilon$-close in the $C^{[\frac{1}{\varepsilon}]}$-topology to the cylindrical
metric $ds^2+d\theta^2$, where $d\theta^2$ denotes the round metric of scalar curvature one on $S^2$.  The number $h_N=R_g(x)^{-1/2}$ is called the scale of the neck. A $(C,\varepsilon)$-cap is a noncompact submanifold $\mathcal{C} \subset M$ diffeomorphic to a 3-ball or to $\re P^3$ minus a ball, with an $\varepsilon$-neck $N \subset \mathcal{C}$ such that $\overline{Y}=\mathcal{C} - N$ is  a compact submanifold with boundary. The boundary $\d \overline{Y}$ of the so-called core $Y$ (interior of $\mathcal{C}-N$)  is required to be the central sphere of some $\varepsilon$-neck in $\mathcal{C}$.  After rescaling to make $R(x)=1$ for some point $x$ in the cap, the diameter, volume, and curvature ratios at any two points are bounded by $C$. A $C$-component is a compact and connected Riemannian manifold $(M^3,g)$ diffeomorphic to $S^3$ or $\re P^3$ , of positive sectional curvature and of bounded geometry controlled by $C$ (after scaling). An $\varepsilon$-round component is a compact and connected Riemannian manifold $(M^3,g)$ such that, after scaling to make $R(x)=1$ for some point $x\in M$, is $\varepsilon$-close in the $C^{[1/\varepsilon]}$-topology to a round metric of scalar curvature one.

It is very important that we can say something more about the geometry of the caps. This kind of information was not needed
in the proof of the Poincar\'{e} Conjecture since the topology of the caps is well-known. We classify the $(C,\varepsilon)$-caps into three types: $A$, $B$ and $C$. The caps of type $A$ have positive sectional curvature everywhere. The caps of type $B$, after scaling,  are $\varepsilon$-close in the $C^{[1/\varepsilon]}$-topology to a fixed metric ball centered at the tip of the standard initial metric. They are diffeomorphic to a 3-ball. Finally, a cap $\mathcal{C}$ of type $C$ comes with a 
double covering $\psi:S^2 \times (-3/\varepsilon-4,3/\varepsilon+4) \g \mathcal{C}$ with $\psi(-\theta,-t)=\psi(\theta,t)$ and such that $h^{-2}\, \psi^*(g)$ is within $\varepsilon$ of $ds^2+d\theta^2$ in the $C^{[1/\varepsilon]}$-topology, where $h=R_g(z)^{-1/2}$ for some $z \in \psi(S^2 \times \{-2/\varepsilon-4\})$. These caps are diffeomorphic to $\re P^3$ minus a ball.

 It follows from Perelman's proof of the
existence of the Ricci flow with surgery that these are the only types of  caps that appear as canonical neighborhoods. We give more details about this fact in Section \ref{section.ricci}.

Another important fact is that when two $\varepsilon$-necks $N$ and $N'$ intersect each other, the ratio $h_{N}/h_{N'}$ between their scales  is very close to 1 (if $\varepsilon$ is sufficiently small) and their product structures almost align. These results are collected in Proposition A.11 of \cite{MORGANTIAN07}, for instance.


The first and key step is to show that any compact orientable 3-manifold $(M,g)$ with the property that every point in $M$ is contained in a canonical neighborhood is isotopic to a canonical metric. This is the content of Proposition \ref{canonical.neighborhood}.

If $(M,g)$ is a $C$-component or an $\varepsilon$-round component, the sectional curvatures are positive. Hence it follows from Hamilton's theorem (\cite{HAMILTON82}) that the normalized Ricci
flow starting at $(M,g)$ has positive sectional curvature and converges to a constant curvature spherical quotient $S^3/\Gamma$. Notice that if $\Gamma$ is nontrivial, a canonical metric on $S^3/\Gamma$, as defined in the introduction, is not round. It is obtained as a Gromov-Lawson connected sum of a round sphere and a round $S^3/\Gamma$. But it is locally conformally flat,
and it follows from the works of Kuiper (\cite{KUIPER49} and \cite{KUIPER50}) that there is a round metric in its conformal class. The conformal method can then be used, as explained in 
Corollary \ref{uniformization}, to connect these two metrics through metrics of positive scalar curvature. The end result is that a $C$-component or an $\varepsilon$-round component 
is always 
isotopic to a canonical metric.

Hence, we can assume that every point $x \in M$ is the center of an $\varepsilon$-neck or is
contained in the core of a $(C,\varepsilon)$-cap. It follows from Propositions A.21 and  A.25 of \cite{MORGANTIAN07} that $M$ is diffeomorphic to $S^3, \re P^3, \re P^3 \# \re P^3,$ or $S^2 \times S^1$. It also follows from their proofs that $M$ is diffeomorphic
to $S^2\times S^1$ if and only if every point $x \in M$ is the center of an $\varepsilon$-neck.

Suppose $M$ is diffeomorphic to $S^3$. The cases in which $M$ is diffeomorphic to $\re P^3$ or $\re P^3 \# \re P^3$ are similar, and the case of $S^2 \times S^1$ will be dealt with later. Hence we  know that $M$ must contain a $(C,\varepsilon)$-cap $\mathcal{C}_1$, which has to be diffeomorphic to a 3-ball.  Let $\psi_{N_1}:S^2 \times (-1/\varepsilon,1/\varepsilon) \g N_1$ be the diffeomorphism associated to the neck $N_1 \subset \mathcal{C}_1$, oriented so that $\psi_{N_1}(\theta,s)$ approaches the boundary $\d \mathcal{C}_1$ as $s \g 1/\varepsilon$. We would like to choose this cap as large as possible. In order to do that we can  ask the question of whether there exists  a point $z \in \mathcal{C}_1$ near the
boundary $\d \mathcal{C}_1$, such that the $s$-coordinate of $\psi_{N_1}^{-1}(z)$ is $0.9/\varepsilon$ for example, with
the property that $z$ is contained in the core of a $(C,\varepsilon)$-cap $\mathcal{C}_2$ with $\mathcal{C}_1\subset \mathcal{C}_2$. In fact we will prove in Section \ref{proofs}, by topological and geometric arguments, that in this case
$\mathcal{C}_1$ is disjoint from the right-hand one-quarter of the neck $N_2 \subset \mathcal{C}_2$.   
If such a point $z$ exists  we replace $\mathcal{C}_1$ by $\mathcal{C}_2$, and ask the same question as before for $\mathcal{C}_2$. Since the scalar curvature of $M$ is bounded below by a positive constant, each quarter of a neck contributes a definite amount to the volume. Since the volume of $M$ is finite, we conclude that there cannot be infinitely many disjoint quarters of  $\varepsilon$-necks, and hence the above process cannot continue forever. 

Therefore there must exist a $(C,\varepsilon)$-cap $\mathcal{C}$, of neck $N$ and core $Y$, such that no point of $\psi_N(S^2 \times \{0.9/\varepsilon\})$ is contained in the core of a $(C,\varepsilon)$-cap that contains
$\mathcal{C}$. Hence any $z_2 \in \psi_N(S^2 \times \{0.9/\varepsilon\})$ is either the center of an $\varepsilon$-neck, or it is contained in the core of a cap  that does not contain $\mathcal{C}$.

If $z_2$ is contained in the core $\tilde{Y}$ of a $(C,\varepsilon)$-cap $\tilde{\mathcal{C}}$, the fact that $\tilde{\mathcal{C}}$   does not contain $\mathcal{C}$ implies that   $S^3 = \mathcal{C} \cup \tilde{\mathcal{C}}$. We prove, moreover, that
the central sphere $\tilde{S}$ of the neck $\tilde{N}\subset \tilde{\mathcal{C}}$ is contained in the region
$\overline{Y} \cup \psi_N(S^2 \times (-1/\varepsilon,0.9/\varepsilon))$. 

More generally we have to consider the possibility of finding some $\varepsilon$-necks. If $z_2$ is the center of an $\varepsilon$-neck $N_2'$, we add $N_2'$ to a list that starts with $N_1'=N$ and  replace $z_2$ by some choice of $z_3 \in \psi_{N_2'}(S^2 \times \{0.9/\varepsilon\})$. By repeating the process we find
 a sequence
$\{N_1'=N, N_2', \dots, N_a'\}$ of $\varepsilon$-necks  with centers $z_1,\dots,z_a$ such that:
\begin{itemize}
 \item [1)]  $z_{i+1} \in \psi_{N_i}(S^2 \times \{0.9/\varepsilon\})$, 
\item [2)]  $N_{i}$ is disjoint from the left-hand one-quarter of $N_1$ for all $1 < i \leq  a$.
\end{itemize}
The necks are all oriented so that  $g\,(\d/\d s_{N_i},\d/\d s_{N_{i+1}})>0$ for all $1 \leq i < a$. A list $\{N_1'=N, N_2', \dots, N_a'\}$ with the above properties is referred to in this paper as a structured chain of $\varepsilon$-necks. Again, for volume reasons, there is an upper  bound on the number of necks in such a chain. This means that we can choose $a$ so that a point $z_{a+1}\in \psi_{N_a}(S^2 \times \{0.9/\varepsilon\})$  is  in the core $\tilde{Y}$ of a $(C,\varepsilon)$-cap $\tilde{\mathcal{C}}$. In that case we prove that 
$$
S^3 = \mathcal{C} \cup N_2' \cup \dots \cup N_a' \cup \tilde{C}.
$$ 


Since we are assuming that $M$ is diffeomorphic to a 3-sphere, the caps $\mathcal{C}$ and $\tilde{\mathcal{C}}$ have to be of type $A$ or $B$. We have to separate the proof into four cases according with the types of $\mathcal{C}$ and $\tilde{\mathcal{C}}$. In order to illustrate the method let us suppose that both caps are of type $A$ (positive sectional curvature) and that $a \geq 2$.

In Section \ref{interpolation.lemmas} we use the fact that the scales of adjacent $\varepsilon$-necks are very close to prove some interpolation lemmas. Given the structured chain of $\varepsilon$-necks $\{N_1'=N, N_2', \dots, N_a'\}$, we
produce a single diffeomorphism $\psi:S^2 \times (-1/\varepsilon, \beta) \g \bigcup_{i=1}^a N_i'$ that coincides with
$\psi_{N_1'}$ on $S^2 \times (-1/\varepsilon,0.25/\varepsilon)$, and with $\psi_{N_a'}\circ T$ on $S^2 \times (\beta -1.25/\varepsilon, \beta)$, where $T$ is some isometry of $ds^2+d\theta^2$. For each neck $N_i'$, because the metrics $h_{N_i'}^{-2}\psi_{N_i'}^*(g)$ and $ds^2 + d\theta^2$ 
are $\varepsilon$-close in the $C^2$ topology, the metrics in the linear homotopy 
$(1-\mu) \psi_{N_i'}^*(g) + \mu h_{N_i'}^2 (ds^2 + d\theta^2)$ have positive scalar curvature for all $\mu \in [0,1]$. By
interpolating between the deformations of each pair of adjacent $\varepsilon$-necks, starting with $N_1'$ and $N_2'$, we can produce a continuous path of metrics $\mu \in [0,1] \g g_\mu$ of  positive scalar curvature
on $S^2 \times (-1/\varepsilon, \beta)$, with
$g_0 = \psi^*(g)$ and  $g_1$ rotationally symmetric, and such that it restricts to the linear homotopy
$g_\mu = (1-\mu)\psi^*(g) + \mu \, h_{N_1'}^2 (ds^2 + d\theta^2)$
on $S^2 \times (-1/\varepsilon,0.25/\varepsilon)$ and  to the linear homotopy
$g_\mu = (1-\mu)\psi^*(g) + \mu \, h_{N_a'}^2  (ds^2 + d\theta^2)$  on $S^2 \times (\beta-1.25/\varepsilon, \beta)$.

We can  now  perform surgery along the central spheres $S_1'=\psi_{N_1'}(S^2 \times \{0\})$ and $S_a'=\psi_{N_a'}(S^2 \times \{0\})$, and glue standard caps to both left and right sides of each sphere as explained in Section \ref{surgery}. In doing this we break the manifold into three 
components: $(\mathcal{S}_1,g_1)$,$(P,g_P)$, and $(\mathcal{S}_2,g_2)$. An important property of surgery is that it
preserves positive sectional curvature. We use this to conclude that both the left-hand
$(\mathcal{S}_1,g_1)$ and the right-hand $(\mathcal{S}_2,g_2)$ components have positive sectional curvature. Therefore they can be deformed to  constant curvature metrics by the normalized Ricci flow.  The middle component $(P,g_P)$  is obtained by attaching standard caps
to the boundary of the region between the spheres $S_1'$ and $S_a'$.  Now we use the fact that the metric deformation on  $\bigcup_{i=1}^a N_i'$  produced by interpolation restricts to linear homotopies
on neighborhoods of the  surgery spheres to conclude that  it can be extended (also linearly) to the attached caps.  This provides a
 deformation of the metric  on $P$, through metrics of positive scalar curvature, that ends in a rotationally symmetric manifold. For that we use  that the standard initial metric is rotationally symmetric. Since any rotationally symmetric manifold is also
locally conformally flat, we can use the conformal method to finish the isotopy of $(P,g_P)$ into a round sphere.

We have proved that each one of the components $(\mathcal{S}_1,g_1)$, $(P,g_P)$, and $(\mathcal{S}_2,g_2)$ is isotopic to a round sphere. Since the Gromov-Lawson connected sum construction can be performed continuously on families of metrics, we obtain that a connected sum $(\mathcal{S}_1,g_1) \,\#\, (P,g_P) \,\#\, (\mathcal{S}_2,g_2)$ is isotopic to a Gromov-Lawson connected sum of three round spheres $S^3_1 \#S^3_2 \# S^3_3$. Again we can use the conformal method (due to
conformal flatness) to prove $S^3_1 \#S^3_2 \# S^3_3$ is isotopic to a single round sphere. The final observation is that
the Gromov-Lawson construction can be used to revert surgery in the following sense: if the necks that formed  $(\mathcal{S}_1,g_1) \,\#\, (P,g_P) \,\#\, (\mathcal{S}_2,g_2)$ were introduced at the surgery tips,  then  a local application of the conformal method at the neck regions implies that $(\mathcal{S}_1,g_1) \,\#\, (P,g_P) \,\#\, (\mathcal{S}_2,g_2)$ is isotopic to the original $(S^3,g)$. This is explained in Section \ref{connected.sums}. 

We have concluded then that 
$(S^3,g)$ is isotopic to a round sphere. The arguments to handle the other types of caps are similar. The key point is that they are almost rotationally symmetric, and therefore can be treated by the conformal method as we did above for $(P,g_P)$. 

If $M$ is diffeomorphic to $S^2 \times S^1$, then every point is contained in an $\varepsilon$-neck whose central sphere does not separate $M$.
Let $N$ be one such neck, with central sphere $S$. We do surgery on $N$ along $S$ and glue standard caps to both sides of it. The resulting manifold is a 3-sphere $(S^3,g_{surg})$ endowed with a metric of positive scalar curvature such that every point of it  has a canonical neighborhood. The previous arguments imply that  $(S^3,g_{surg})$ is isotopic to a round sphere. Again we conclude that  the original manifold $(S^2 \times S^1,g)$ is isotopic to the Gromov-Lawson connected sum of 
$(S^3,g_{surg})$ with itself, where the connected sum is performed at the tips of the spherical caps. By continuously
performing the connected sum of $S^3$ to itself along the isotopy of $(S^3,g_{surg})$, we conclude that $(S^2 \times S^1,g)$ is isotopic to   a Gromov-Lawson connected sum of a round 3-sphere to itself. But this is the exact definition of
a canonical metric on $S^2\times S^1$.

This finishes the  first step:  to show that any compact orientable 3-manifold $(M,g)$ such that every point in $M$ is contained in a canonical neighborhood is isotopic to a canonical metric. 

Let us go back to the Ricci flow with surgery
$(M^3_i,g_i(t))_{t \in [t_i,t_{i+1})}$ with initial condition $(M^3,g_0)$. 
 The manifold $(M_j,g_{t_j})$ is clearly isotopic to $(M_j,g_{t'})$, through Ricci flow.
Since each component of $(M_j,g_{t'})$ is such that every point has a canonical neighborhood, we conclude that
every component of $M_j$ at time $t_j$ is isotopic to a canonical metric.

It remains to explain the induction step: to prove that if every component of $M_{i+1}$ at time $t_{i+1}$ is isotopic to a canonical
metric, then the same is true for any component of $M_i$ at time $t_i$. We will give the details of that in Section 
\ref{proofs}. It follows from Perelman's description of singularity formation, and arguments similar to the ones we
sketched above, that any component of $M_{i}$ at time $t_{i+1}-\eta$, with $\eta >0$ sufficiently small, is isotopic
to a connected sum of some of the components of $M_{i+1}$ at time $t_{i+1}$ and finitely many components $(P_j,g_{P_j})$ that
are covered by canonical neighborhoods. The existence of the components $(P_j,g_{P_j})$  comes from the high curvature regions that are discarded after surgery. It follows from the step one of the induction that the manifolds
$(P_j,g_{P_j})$ are isotopic to canonical metrics, and it follows from  the induction hypothesis that the components of $M_{i+1}$ at time $t_{i+1}$ are isotopic to canonical metrics as well. The theorem will follow once we prove that any connected manifold obtained from  finitely many components endowed with canonical metrics by performing Gromov-Lawson connected sums  is isotopic to  a single component with a canonical metric. This is the content of Lemma \ref{canonical.metrics}.

It follows from backwards induction on $i$ that  any metric of positive scalar curvature on $M^3$
can be continuosly deformed, through metrics of positive scalar curvature, into a canonical metric. This proves that the moduli space
$\mathcal{R}_+(M)/{\rm Diff}(M)$ is path-connected.


\section{Conformal deformations}\label{conformal}

This section will present a few applications of the  conformal method.

\begin{prop}\label{conformal.deformations}
Let $(M^n,g)$ be a compact Riemannian manifold. Then the space  of metrics with positive scalar curvature in the   conformal class $[g]$ of $g$ is contractible. 
\end{prop}

\begin{proof}
If $n \geq 3$ the set
$$
\{u \in C^\infty(M): u>0,  R_{u^\frac{4}{n-2}g}>0\}
$$
is convex, since $ R_{u^\frac{4}{n-2}g} = u^{-\frac{n+2}{n-2}}\, \Big(- \frac{4(n-1)}{(n-2)}\Delta_g u + R_gu\Big)$. 


Similarly, if $n=2$ the set
$$
\{f \in C^\infty(M):   K(e^{f} g)>0\} 
$$
is convex, since $K_{e^f\,g} = e^{-f}\, \Big(K_g - 1/2\, \Delta_g f\Big)$. Here $K_g$ denotes the Gauss curvature of the metric $g$.

 The proof now is straightforward since these sets parametrize the metrics of positive scalar curvature in $[g]$.
\end{proof}

Hence
\begin{cor}\label{uniformization}
Let $\Gamma$ be a finite subgroup of $O(n+1)$ acting freely 
on $S^n$, $n \geq 3$. Suppose $g$ is a locally conformally flat metric on $S^n/\Gamma$ with positive scalar curvature. Then there exists
a continuous path of metrics $g_\mu= e^{f_\mu}\, g$ of positive scalar curvature, $\mu \in [0,1]$, such that $g_0=g$ and $g_1$ has
constant sectional curvature.
\end{cor}

\begin{proof}
 Since $g$ is locally conformally flat, it follows from the works of Kuiper (\cite{KUIPER49} and \cite{KUIPER50}) that  there exists a  metric $\overline{g}$ of constant sectional curvature  in the
conformal class of $g$  (see also   \cite{SCHOENYAU88}, \cite{SCHOENYAU94} and \cite{TANNO73}). 

The corollary then follows from Proposition \ref{conformal.deformations}.
\end{proof}

Let $d\theta^2$ be a constant curvature metric on $S^{n-1}$ of scalar curvature one.
The next application  will be used later to deform surgery necks.
\begin{prop}\label{symmetric.deformation}
Let $g=dr^2 + w^2(r) \, d\theta^2$ be a rotationally symmetric Riemannian metric on $S^{n-1} \times (a,b)$, $n \geq 3$, of positive
scalar curvature. Assume $w = 1$ in the intervals $(a,a')$ and $(b',b)$.  Then there exists a continuous
path $\mu \in [0,1]   \mapsto g_\mu$ of rotationally symmetric positive scalar curvature metrics on $S^{n-1} \times (a,b)$
such that $g_0=g$, $g_1=dr^2 + d\theta^2$, and $g_\mu = g$ in $S^{n-1} \times (a,a')$ and $S^{n-1} \times (b',b)$ for all $\mu \in [0,1]$.
\end{prop}

\begin{proof}
Let $\tilde{g} = w^{-2}(r) \, g$, and $v = \int w^{-1}(r)\, dr$.  Then $\tilde{g} = dv^2 + d\theta^2$. 

If $\mu \in [0,1/2]$, we define $g_\mu = \Big((1-2\mu) + 2\mu\, w^{\frac{2-n}{2}}(r)\Big)^\frac{4}{n-2}\, g$. Hence $g_0=g$, $g_{1/2}=\tilde{g}$, and $g_\mu = g$ in $(a,a')$ and $(b',b)$ for all $\mu \in [0,1/2]$. It follows from the proof of Proposition 
\ref{conformal.deformations} that $g_\mu$ has 
 positive scalar curvature for all $\mu \in [0,1/2]$.

If $\mu \in [1/2,1]$, we define $g_\mu = \Big((2-2\mu)\, w^{-2}(r) + 2\mu -1\Big) dr^2 + d\theta^2$. Hence $g_{1/2}=\tilde{g}$, $g_1= dr^2 + d\theta^2$, and $g_\mu = g$ in $(a,a')$ and $(b',b)$ for all $\mu \in [1/2,1]$. It is easy to see by a change of 
variables that the metrics $g_\mu$ are cylindrical for all $\mu \in [1/2,1]$.   This completes
the proof of the proposition.
\end{proof}


\section{Interpolation Lemmas}\label{interpolation.lemmas}

We will define the concept of an $\varepsilon$-neck structure and prove some  interpolation lemmas. We refer the reader
to the appendix of \cite{MORGANTIAN07} for basic properties of $\varepsilon$-necks and their intersections. 

Let $(M^3,g)$ be a Riemannian manifold. Given $\varepsilon>0$, an {\it $\varepsilon$-neck structure} on an open set $N \subset M$, centered at $x \in M$, is a diffeomorphism 
$\psi: S^2 \times (-1/\varepsilon, 1/\varepsilon) \g N \subset M$ with $x \in \psi\, (S^2 \times \{0\})$, and such that
$R_g(x)\psi^*(g)$ is within $\varepsilon$ in the $C^{[1/\varepsilon]}$-topology of the product metric $g_{cyl}= ds^2+ d\theta^2$, where
$d\theta^2$ is a fixed metric on $S^2$ of constant scalar curvature 1. We call $N$ an {\it $\varepsilon$-neck} of 
{\it central sphere} $S_N=\psi\, (S^2 \times \{0\})$ and  {\it scale}  $h_N=R_g(x)^{-1/2}$. Let also 
$s_N: N \g \mathbb{R}$ be the function  $s_N(\psi(\theta,t))=t$, and $\d/\d s_N = \psi_*(\d/\d s)$.


\begin{lem}\label{interpolation.lemma.1}
 There exists $0<\varepsilon_1 \leq 1$ such that the following is true. Suppose $0 < \varepsilon \leq \varepsilon_1$, and let $\psi_1: S^2 \times (-1/\varepsilon, 1/\varepsilon) \g N_1$ and $\psi_2: S^2 \times (-1/\varepsilon, 1/\varepsilon) \g N_2$ be
$\varepsilon$-neck structures  in $(M^3,g)$ of scales $h_1$ and $h_2$, respectively,  such that  
$$
\Lambda= s_{N_1}^{-1}((-0.95/\varepsilon,0.95/\varepsilon))\cap s_{N_2}^{-1}((-0.95/\varepsilon,0.95/\varepsilon)) \ne \emptyset.
$$ 
Suppose also that any embedded 2-sphere separates the manifold $M$.  Let
$z\in \Lambda$, and assume  $g\, (\d/\d s_{N_1},\d/\d s_{N_2})>0$. Then there exists a diffeomorphism 
$\psi:S^2 \times (-1/\varepsilon,\beta) \g \tilde{N} \subset N_1 \cup N_2$, $\beta=1/\varepsilon-s_{N_2}(z)+s_{N_1}(z)$, with the following properties:
\begin{itemize}
 \item [a)] $\psi\,(\theta,t)=\psi_1(\theta,t)$ for   $(\theta,t) \in S^2 \times (-1/\varepsilon,s_{N_1}(z)-0.025/\varepsilon)$,
 \item [b)] $\psi\,(\theta,t)=\psi_2(A\cdot \theta, t+1/\varepsilon-\beta)$ for 
$(\theta,t) \in S^2 \times (\beta +s_{N_2}(z)-0.975/\varepsilon,\beta)$, where $A$ is an isometry of $(S^2,d\theta^2)$,
\item [c)] there exists  a continuous path of metrics $\mu \in [0,1] \mapsto  g_\mu$ of positive scalar curvature on $S^2 \times (-1/\varepsilon,\beta)$, with
$g_0 = \psi^*(g)$, $g_1$ rotationally symmetric, and such that it restricts to the linear homotopy
$g_\mu = (1-\mu)\psi^*(g) + \mu \, h_1^2 g_{cyl}$
on $S^2 \times (-1/\varepsilon,s_{N_1}(z)-0.025/\varepsilon)$ and to the linear homotopy
$g_\mu = (1-\mu)\psi^*(g) + \mu \, h_2^2 g_{cyl}$  on $S^2 \times (\beta +s_{N_2}(z)-0.975/\varepsilon,\beta)$.
\end{itemize}
\end{lem} 

\begin{figure}
\begin{center}
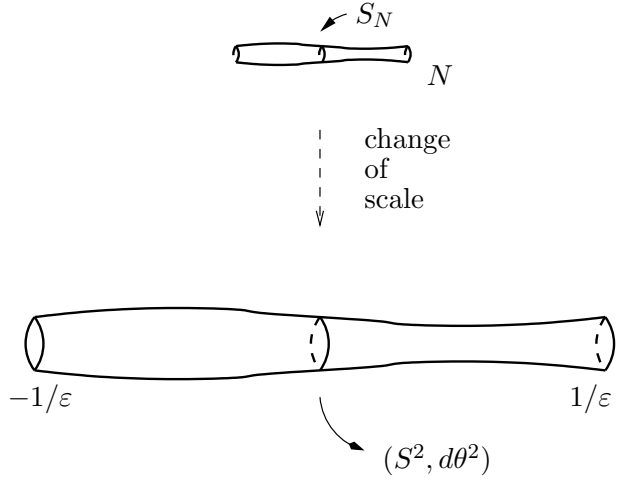
\caption{An $\varepsilon$-neck after scaling.}
\end{center}
\end{figure}

\begin{proof}
The Proposition A.11 of \cite{MORGANTIAN07} asserts that  $|h_1/h_2-1|< 0.1$ if
 $\varepsilon_1$ is sufficiently small. It is not difficult to see then that, if $\varepsilon_1$ is sufficiently small, we have
$$
\psi_1(S^2 \times (s_{N_1}(z)-4,s_{N_1}(z)+4))\subset \psi_2(S^2 \times (s_{N_2}(z)-0.04/\varepsilon,s_{N_2}(z)+0.04/\varepsilon)).
$$

Let  
$$
\varphi: S^2 \times (s_{N_1}(z)-4,s_{N_1}(z)+4) \g S^2 \times (s_{N_2}(z)-0.04/\varepsilon,s_{N_2}(z)+0.04/\varepsilon)
$$ be given by 
$\varphi= \psi_2^{-1} \circ \psi_1$. Notice that $\varphi$ is an isometry between the metrics $h_2^{-2}\psi_2^*(g)$ and $h_2^{-2}\psi_1^*(g)$, which are both small perturbations of $g_{cyl}$. Given any $\alpha >0$ we can choose $\varepsilon_1$
sufficiently small 
so that there is always an isometry $A$ of $(S^2,d\theta^2)$ with
$\varphi$ within $\alpha$ in the $C^{[1/\alpha]}$-topology over $S^2 \times (s_{N_1}(z)-3,s_{N_1}(z)+3)$ of the map $\hat{\varphi}$ given by
$\hat{\varphi}(\theta,t) = (A\cdot \theta, t + s_{N_2}(z)-s_{N_1}(z))$. It is not difficult to prove the existence of $A$ by
a contradiction argument.

Write $\hat{\varphi}^{-1} \circ \varphi\, (\theta,t) = (p(\theta,t),t + q(\theta,t)) \in S^2 \times \mathbb{R}$, where
$p(\theta,t) = \exp_\theta V(\theta,t)$, $V(\theta,t) \in T_\theta S^2$. By $\exp$ we mean  the exponential map of $(S^2,d\theta^2)$.

Let   $\eta : \re \g \re$ be a cutoff function such that $0 \leq \eta \leq 1$, $\eta(t) = 1$ if $t \leq -1$, and
$\eta(t)=0$ if $t \geq 1$. Define 
$$
\gamma(\theta,t) = \Big(\exp_\theta (\eta(t-s_{N_1}(z))\, V(\theta,t)), t+ \eta(t-s_{N_1}(z))q(\theta,t)\Big).
$$ 
Hence $\gamma(\theta,t) = \hat{\varphi}^{-1} \circ \varphi(\theta,t)$ if $t \leq s_{N_1}(z)-1$, while $\gamma(\theta,t)=(\theta,t)$ if
$t \geq s_{N_1}(z)+ 1$. Notice that $\gamma$ is a small perturbation of the identity map. 


Define $\psi : S^2 \times (-1/\varepsilon,\beta) \g M$ by 
$$
\psi(\theta,t) = 
\left\{
\begin{array}{lcr}
\psi_1(\theta,t) &\ {\rm if } \ &-1/\varepsilon < t \leq s_{N_1}(z)-2,\\
\psi_2(\hat{\varphi}(\gamma(\theta,t))) &\ {\rm if } \ &s_{N_1}(z)-2\leq t < s_{N_1}(z)+2,\\
\psi_2(\hat{\varphi}(\theta,t)) &\ {\rm if } \ &s_{N_1}(z)+2 \leq t < \beta.
\end{array}
\right.
$$
Since any embedded 2-sphere separates the manifold $M$, we necessarily have that 
$\psi_1(S^2 \times (-1/\varepsilon,s_{N_1}(z)-2)) \cap \psi_2(S^2 \times (s_{N_2}(z)+2,1/\varepsilon)) =\emptyset$. If
$\varepsilon_1$ is sufficiently small, it follows that  $\psi$ is a diffeomorphism onto its image, and  it satisfies properties (a)
and (b) of the Lemma.

Let $\hat{g}$ be the  rotationally symmetric metric given by
$$
\hat{g}(\theta,t) = \Big(\eta(t-s_{N_1}(z)) h_1^2 + (1-\eta(t-s_{N_1}(z)))h_2^2\Big) g_{cyl}(\theta,t).
$$
Hence
$$
\hat{g}(\theta,t) = 
\left\{
\begin{array}{lcr}
h_1^2g_{cyl}(\theta,t) &\ {\rm if } \ &-1/\varepsilon < t \leq s_{N_1}(z)-2,\\
h_2^2g_{cyl}(\theta,t) &\ {\rm if } \ &s_{N_1}(z)+2 \leq t < \beta.
\end{array}
\right.
$$

Since  $h_1^{-2}\hat{g}$ and $h_1^{-2} \psi^*(g)$ are both small perturbations of $g_{cyl}$,   the metrics 
$$
(1-\mu) \psi^*(g) + \mu \hat{g}
$$ 
have positive scalar curvature for all $\mu \in [0,1]$. This proves property (c).
\end{proof}

A {\it structured chain of $\varepsilon$-necks} in $(M^3,g)$ is a sequence
$\{N_1, \dots, N_a\}$ of $\varepsilon$-necks  with centers $x_1,\dots,x_a$ such that:
\begin{itemize}
 \item [1)]  $s_i(x_{i+1})=0.9/\varepsilon$ and $g\,(\d/\d s_{N_i},\d/\d s_{N_{i+1}})>0$ for all $1 \leq i < a$,
\item [2)]  $N_{i}$ is disjoint from the left-hand one-quarter of $N_1$ for all $1 < i \leq  a$.
\end{itemize}

\begin{figure}
\begin{center}
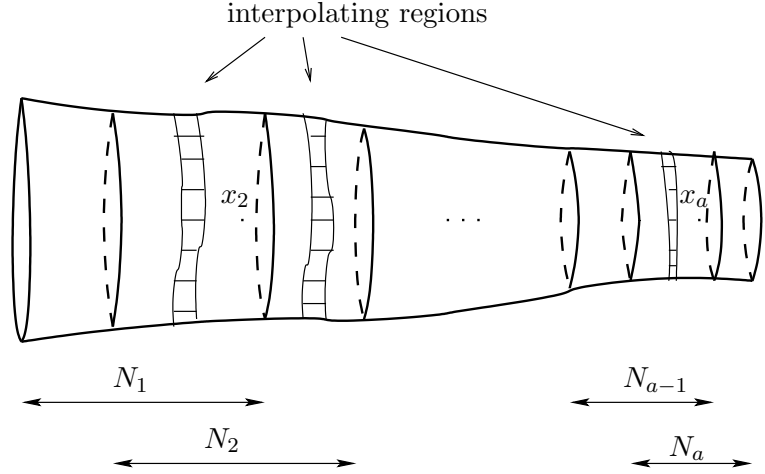
\caption{A structured chain of $\varepsilon$-necks.}
\end{center}
\end{figure}

\begin{lem}\label{interpolation.lemma.2}
 Let $\{N_1,\dots,N_a\}$ be a structured chain of $\varepsilon$-necks  in $(M^3,g)$ of scales $h_1,\dots,h_a$. If
 $0<\varepsilon \leq \varepsilon_1$, then 
there exists a diffeomorphism
$\psi: S^2 \times (-1/\varepsilon, \beta) \g \bigcup_{i=1}^a N_i$, $\beta>1.5/\varepsilon$, with the following properties:
\begin{itemize}
 \item [a)] $\psi\,(\theta,t) = \psi_1(\theta,t)$ for $(\theta,t) \in S^2 \times (-1/\varepsilon,0.25/\varepsilon)$,
\item [b)] $\psi\, (\theta,t) = \psi_a(A\cdot \theta,t-\beta+1/\varepsilon)$ for $(\theta, t) \in S^2 \times (\beta-1.25/\varepsilon, \beta)$ and some isometry  $A$ of $(S^2,d\theta^2)$,
\item[c)] there exists  a continuous path of metrics $\mu \in [0,1] \mapsto g_\mu$ 
\end{itemize}
\end{lem}

\begin{proof} If one applies the Lemma \ref{interpolation.lemma.1} to the necks $N_i$ and $N_{i+1}$, with 
$z_i \in s_{N_i}^{-1}(0.5/\varepsilon)$,
 the interpolation takes place in a region contained in the 
intersection of the right-hand half of $N_i$ and the left-hand half of $N_{i+1}$. The left-hand half of $N_i$ stays parametrized
by $\psi_i$, while the right-hand half of $N_{i+1}$ becomes parametrized by the composition of $\psi_{i+1}$ with an isometry of
$(S^2 \times \re, g_{cyl})$. It is then clear that the diffeomorphisms given by Lemma \ref{interpolation.lemma.1} can be matched together. The proof follows easily from Lemma \ref{interpolation.lemma.1}.
\end{proof}


\section{Surgery}\label{surgery}

In this section we will describe the basic properties of the surgery process. For more details
see \cite{PERELMAN03A} and Chapter 13 of \cite{MORGANTIAN07} (compare \cite{CAOZHU06} and \cite{KLEINERLOTT08}). 

The {\it standard initial metric}  is a  complete metric $g_{std}$ on $\re^3$ with the following properties:
\begin{itemize}
 \item [1)] $g_{std}$ has nonnegative sectional curvature,
\item[2)] $g_{std}$ is rotationally symmetric, i.e., invariant under the usual $SO(3)$-action on $\re^3$,
\item[3)] there exists $A_0 > 0$ such that $(\re^3 \setminus B(0,A_0),g_{std})$ is isometric to the cylindrical metric
of scalar curvature one on $S^2 \times (-\infty,4]$,
\item[4)] there exists $r_0 > 0$ such that $(B(0,r_0),g_{std})$ is isometric to a ball of radius $r_0$ inside  a 3-sphere of radius 2,
\item[5)] there exists $\beta > 0$ such that the scalar curvature satisfies $ R_{g_{std}}(x) \geq \beta$ for every $x \in \re^3$.
\end{itemize}

\begin{figure}
\begin{center}
\resizebox{7cm}{!}{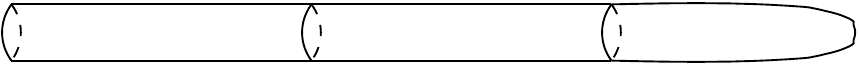} 
\caption{The standard initial metric.}
\end{center}
\end{figure}

The fixed point of the action of $SO(3)$ (the origin $0 \in \re^3$)  is called the {\it tip} of the standard initial metric. Let $s': \re^3 \setminus B(0,A_0) \g (-\infty,4]$ be the projection onto
the second factor through the isometry $\psi':S^2 \times (-\infty,4] \g \re^3 \setminus B(0,A_0)$ mentioned in the item (3) above. We extend it to a map $s': \re^3 \g (-\infty,4+A_0]$
by $s'(x)= 4+A_0-d_{g_{std}}(x,p)$. This map is an isometry along each radial geodesic emanating from $p$. It is smooth except
at $p$, and $s'(p)=4+A_0$. The pre-images of $s'$ are round 2-spheres.

Let $h$ be a Riemannian metric on $S^2 \times (-4,4)$ such that $h$ is within $\varepsilon$ of the cylindrical metric $ds^2+d\theta^2$
in the $C^{[1/\varepsilon]}$-topology. The smooth manifold
$\mathcal{S}$ is obtained by gluing together $S^2 \times (-4,4)$ and $B(p,A_0+4)$ and identifying
$(x,s)$ with $\psi'(x,s)$ for all $x\in S^2$ and $s \in (0,4)$.

Define
$$
f_\varepsilon(s) = 
\left\{
\begin{array}{lcr}
0 &\ {\rm if } \ &s \leq 0\\
C\,\varepsilon\, e^{-q/s} &\ {\rm if } \ &s > 0,
\end{array}
\right.
$$
where $C$ and $q$ will be chosen later independently of $\varepsilon$.

Let $\alpha:[1,2] \g [0,1]$ and $\beta:[4+A_0-r_0,4+A_0] \g [0,1]$ be cutoff functions such that $\alpha$ is identically 1 near 1 and
identically 0 near 2, while $\beta$ is identically 1 near $4+A_0-r_0$ and identically 0 near $4+A_0$. Set $\lambda = \sqrt{1-\varepsilon}$.

The metric $h_{surg,\varepsilon}$  is defined on $\mathcal{S}$  by 
$$
\left\{
\begin{array}{lcr}
e^{-2f_\varepsilon(s)}h &\ {\rm on } \ &s^{-1}((-4,1])\\
e^{-2f_\varepsilon(s)}\Big(\alpha(s)\,h + (1-\alpha(s))\lambda\, g_{std}\Big) &\ {\rm on } \ &s^{-1}([1,2])\\
e^{-2f_\varepsilon(s)}\lambda\,  g_{std} &\ {\rm on } \ &s^{-1}([2,A_{r_0}])\\
\Big(\beta(s)e^{-2f_\varepsilon(s)} + (1-\beta(s))e^{-2f_\varepsilon(4+A_0)}\Big)\lambda\, g_{std} &\ {\rm on } \ &s^{-1}([A_{r_0},A']),
\end{array}
\right.
$$
where $A_{r_0}=4+A_0-r_0$ and $A'=A_0+4$. Notice that $h_{surg,\varepsilon}$ coincides with $h$ on $s^{-1}((-4,0])$. It is
also clear that $h_{surg,\varepsilon}$ and $g_{std}$ get arbitrarily close as we consider $\varepsilon$-necks with $\varepsilon \g 0$. The process of going from $(S^2 \times (-4,4),h)$ to $(\mathcal{S},h_{surg,\varepsilon})$ will be  called {\it surgery}
(or {\it $\varepsilon$-surgery}).

\begin{figure}
\begin{center}
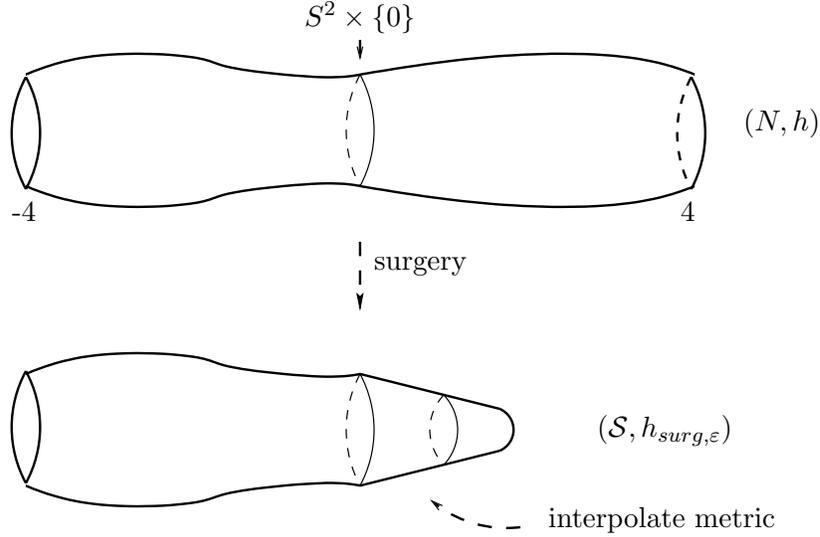
\caption{Neck surgery.}
\end{center}
\end{figure}

The next theorem collects some of the properties proved  in Chapter 13 of \cite{MORGANTIAN07}.
\begin{thm}
There are constants $C,q < \infty$, and $0 < \varepsilon_2 \leq \varepsilon_1$, such that the following hold for $h_{surg}=h_{surg,\varepsilon}$, if $h$ is within $\varepsilon$ of the cylindrical
metric $g_{cyl}$ in the $C^{[1/\varepsilon]}$-topology, $0<\varepsilon \leq \varepsilon_2$:
\begin{itemize}
 \item [a)]  the restriction of $h_{surg}$ to $s^{-1}([1,4+A_0])$ has positive sectional curvature,
\item[b)] the scalar curvature of $h_{surg}$  satisfies $R_{h_{surg}} \geq R_h$ on $s^{-1}((-4, 1])$,
\item[c)] the smallest eigenvalue of the curvature operator $Rm_{h_{surg}}$ is greater than or equal to the smallest eigenvalue of $Rm_h$ at
any point in $s^{-1}((-4, 1])$.
\end{itemize}
\end{thm}

Therefore
\begin{cor}\label{positive.curvature.surgery}
The metric $h_{surg,\varepsilon}$ has positive scalar curvature. If $h$ has positive sectional curvature, so does $h_{surg,\varepsilon}$.
\end{cor}


Notice that if  $0<\varepsilon\leq \tilde{\varepsilon} \leq \varepsilon_2$, and $\varepsilon_\mu = (1-\mu)\, \varepsilon + \mu\, \tilde{\varepsilon}$, then
$\mu \in [0,1] \mapsto h_{surg,\varepsilon_\mu}$ is a continuous family of positive scalar curvature metrics on $\mathcal{S}$  which all coincide with $h$
on $s^{-1}((-4,0))$. 


We will need a lemma to deform surgery caps.
\begin{lem}\label{deforming.surgery}
 Let $h$ be a Riemannian metric on $S^2 \times (-4,4)$ such that $h$ is within $\varepsilon$ of the cylindrical metric $ds^2+d\theta^2$
in the $C^{[1/\varepsilon]}$-topology, $0 < \varepsilon \leq \varepsilon_2$. 
  Then there exists a continuous path of metrics $\mu \in [0,1] \mapsto h_\mu'$ of positive scalar curvature on $\mathcal{S}$ with $h_0'=h_{surg,\varepsilon}$, $h_1'$ rotationally symmetric, and such that
 it restricts to  the linear homotopy 
$h_\mu'=(1-\mu)\, h+\mu\, g_{cyl}$ on $s^{-1}((-4,0))$ for all $\mu \in [0,1]$.
\end{lem}

\begin{proof} If  $h_\mu=(1-\mu)\, h+ \mu\, g_{cyl}$, define
$$
h_\mu' = (h_\mu)_{surg,{\varepsilon}}
$$
for $\mu \in [0,1]$. It is clear from the definition   that $(g_{cyl})_{surg,\varepsilon}$ is rotationally symmetric.
\end{proof}



\section{Connected sums}\label{connected.sums}

In this section we will  discuss the connected sum construction of positive scalar curvature metrics due to Gromov and Lawson (see \cite{GROMOVLAWSON80}).  We will then prove some deformation results  which will be used later.

Let $(M^n,g)$ be a Riemannian manifold of positive scalar curvature. Given $p \in M$,  $\{e_k\} \subset T_pM$ an orthonormal basis,   and $r_0 >0$, we
will define a positive scalar curvature  metric $g'$ on $B_{r_0}(p)\setminus \{p\}$ 
that coincides with $g$ near the boundary
$\d B_{r_0}(p)$, and such that $(B_{r_2}(p) \setminus \{p\},g')$ is isometric to a half-cylinder for some $r_2>0$. 


The construction uses a carefully chosen planar curve $\gamma \subset \re^2$. We  identify $B_{r_0}(p)$ with $B_{r_0}(0) \subset \re^n$
through the choice of $\{e_k\}$ and  exponential normal coordinates. It follows from \cite{GROMOVLAWSON80} that,
 given $0< r_0 \leq \min\{ \frac12 {\rm inj}_M(p), 1\}$, a positive lower bound $\delta$ for the scalar curvature of $g$,   and an upper bound 
$1/\eta>0$ for
the $C^2$ norm of $g$ in exponential coordinates about $p$,  there exists a planar curve $\gamma=\gamma(r_0,\delta,\eta)$ with the following properties:
\begin{itemize}
 \item [1)] the image of $\gamma$ is contained in the region $\{(r,t):r\geq 0, t \geq 0\}$,
\item [2)] the image of $\gamma$ contains  the horizontal half-line $r \geq r_1, t=0$  for some $0<r_1 < r_0$,
\item[3)] the image of $\gamma$ contains  the vertical half-line $r=r_2,t \geq t_2$  for some  $0<r_2 <r_1$ and $t_2 >0$,
\item[4)] the  induced metric on $M' = \{(x,t): (|x|,t) \in \gamma\}$ as a submanifold of  the Riemannian product $B_{r_0}(p) \times \re$ has positive scalar curvature.
\end{itemize}

Since  $r_2$ is  small,  the induced metric  on the tubular piece $r=r_2, t \geq t_2$
is a perturbation of the cylindrical metric on $S^{n-1}_{r_2}(0) \times \re$, where $S^{n-1}_{r_2}(0) \subset \re^n$ is the standard sphere
of radius $r_2$. We can  slightly modify it with a cut-off function to achieve a positive scalar curvature metric $g'$ of the form
$g_{ij}(x,t)\, dx_i\, dx_j + dt^2$, which coincides with the original metric near $t_2$ and such that it is isometric to 
 $S^{n-1}_{r_2}(0) \times [t_3,\infty)$ for some $t_3>t_2$ and $t \geq t_3$.

Let $(M^n_1,g_1)$ and $(M^n_2,g_2)$ be compact manifolds of positive scalar curvature. Given  $p_1 \in M_1$, $p_2 \in M_2$, and
orthonormal bases $\{e_k\} \subset T_{p_1}M_1$, $\{\overline{e}_k\} \subset T_{p_2}M_2$, there exist 
$0 < r_0 \leq  \min\{\frac12 {\rm inj}_{M_1}, \frac12  {\rm inj}_{M_2},1\}$, $\delta >0$, and $\eta>0$ so that the previous construction
with $\gamma=\gamma(r_0,\delta,\eta)$ 
applies to both manifolds.  We can glue together the manifolds $(B_{r_0}(p_1)\setminus \{p_1\},g_1')$ and
 $(B_{r_0}(p_2)\setminus \{p_2\},g_2')$ along the spheres where $t=t_3+1$, with reverse orientations. The result is a  
 positive scalar curvature metric $g_1 \# g_2$ , depending only on $g_1$, $g_2$,  and the  choice of parameters,  on the connected sum $M_1 \# M_2$.
 
 \begin{figure}
\begin{center}
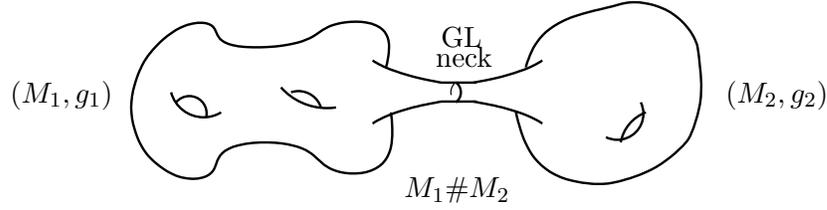
\caption{Gromov-Lawson connected sum.}
\end{center}
\end{figure}

For the same reasons this construction can be applied to families:
\begin{prop}\label{connected.sum.continuity}
Let $\mu \in [0,1] \mapsto g_{i,\mu}$ be  continuous paths of positive scalar curvature metrics  on the compact
manifolds $M^n_i$, $i=1,2$. Given continuous choices of points $\mu \in [0,1] \mapsto p_{i,\mu} \in M_i$, and  of orthonormal bases $\mu \in [0,1] \mapsto \{e_k^{(i)}(\mu)\}$
of $(T_{p_{i,\mu}}M_i,g_{i,\mu})$, $i=1,2$, there exist $r_0>0$, $\delta>0$, and $\eta>0$   such that 
the positive scalar curvature connected sums  
$$
(g_{1,\mu} \# g_{2,\mu})_{\mu \in [0,1]},
$$ constructed with $\gamma=\gamma(r_0,\delta,\eta)$ and $\{e_k^{(i)}(\mu)\}$ at $p_{i,\mu}$, form a continuous path on $M_1 \# M_2$. 
These metrics are such that $g_{1,\mu} \# g_{2,\mu} = g_{1,\mu}$ on $M_1 \setminus B_{r_0}(p_{1,\mu})$, and $g_{1,\mu} \# g_{2,\mu} = g_{2,\mu}$ on $M_2 \setminus B_{r_0}(p_{2,\mu})$
for every $\mu \in [0,1]$.
\end{prop}

It is also possible to perform connected sums of $M$ to itself. This procedure is referred to as {\it attaching a handle}
to $M$. 

{\bf Remark:} If the metric $g$ has constant positive sectional curvature on $B_{r_0}(p)$, then the
metric $g'$ is rotationally symmetric on $B_{r_0}(p) \setminus \{0\}$. Therefore any connected sum of round spheres (or space forms, in general) 
is locally conformally flat. 

Suppose that $B$ and $B'$ are two disjoint copies of the constant curvature ball $B_{r_0}(p)$, and let  $\gamma_1$ and $\gamma_2$ be planar curves as above. It is not difficult to see that, by increasing the length of one of the cylindrical necks, the resulting metrics on $B \#_{\gamma_1} B'$ and $B \#_{\gamma_2} B'$ are pointwise conformal to each other and coincide near the ends. Here the connected sums are performed at the centers.  It follows from the conformal method, like
in Proposition \ref{conformal.deformations}, that the connected sum $B \#_{\gamma_1} B'$ can be deformed, through metrics of positive scalar curvature, into $B \#_{\gamma_2} B'$, without changing the metric  near the boundary. 



Let $h$ be a Riemannian metric on $S^2 \times (-4,4)$  which is within $\varepsilon$ of the cylindrical
metric $g_{cyl}=ds^2+ d\theta^2$ in the $C^{[1/\varepsilon]}$-topology. Let $(\mathcal{S}^{-},h_{surg,\varepsilon}^-)$ and $(\mathcal{S}^{+},h_{surg,\varepsilon}^+)$ be
the manifolds obtained from this neck by doing $\varepsilon$-surgery along the central sphere $S^2 \times \{0\}$, and 
gluing standard caps to both $S^2 \times (-4,0]$ and $S^2 \times [0,4)$, respectively. Their tips $p^-,p^+$ have neighborhoods which are isometric to a geodesic ball   in some standard sphere. We can  apply the Gromov-Lawson connected sum construction to these balls at  $p^-,p^+$,  with some  choice of
parameters $r_0, \delta$, and $\eta$, obtaining a connected sum $(\mathcal{S}^- \#\, \mathcal{S}^+,h_{surg,\varepsilon}^- \# h_{surg,\varepsilon}^+)$ of
positive scalar curvature.

\begin{lem}\label{cylinder.deformation}
Let $g_{cyl}=ds^2 + d\theta^2 $ be the cylindrical metric on $S^2 \times (-4,4)$. Given
$0< \varepsilon \leq \varepsilon_2$, the manifold $(\mathcal{S}^- \#\, \mathcal{S}^+,(g_{cyl})_{surg,\varepsilon}^- \# (g_{cyl})_{surg,\varepsilon}^+)$ can
be continuously deformed into $(S^2 \times (-4,4),g_{cyl})$ through metrics of
positive scalar curvature which all coincide with $g_{cyl}$ on the regions $s^{-1}((-4,-1))$
and $s^{-1}((1,4))$.
\end{lem}

\begin{proof} Notice that $(g_{cyl})_{surg,\varepsilon}^-$ and $(g_{cyl})_{surg,\varepsilon}^+$ are rotationally symmetric. Since the connected sum is performed to small constant curvature balls centered at the tips, we can identify
 $(\mathcal{S}^- \#\, \mathcal{S}^+,(g_{cyl})_{surg,\varepsilon}^- \# (g_{cyl})_{surg,\varepsilon}^+)$  with
$$
\Big(S^2 \times (-4-a,a+4), dr^2 + w^2(r)\,d\theta^2\Big)
$$
 for some $a>0$, and $w(r)>0$ with   $w(r)=1$ if $r \in (-4-a,-a) \cup (a,a+4)$. 

Let $\beta:(-4-a,a+4) \g (-4,4)$ be a diffeomorphism such that $\beta(r)=r+a$ if $r \in (-4-a,-1-a)$, and
$\beta(r)=r-a$ if $r \in (a+1,a+4)$. The result follows by applying the Proposition
\ref{symmetric.deformation} to $(id,\beta)_*(dr^2 + w^2(r)\,d\theta^2)$ on $S^2 \times (-4,4)$.
\end{proof}

\begin{lem}\label{undoing.surgery}
 There exists $0<\varepsilon_3 \leq  \varepsilon_2$ such that the following is true. 
 If $h$ is within $\varepsilon$ of the cylindrical
metric $g_{cyl}$ on $S^2 \times (-4,4)$  in the $C^{[1/\varepsilon]}$-topology,  $0<\varepsilon\leq \varepsilon_3$, then
  $(\mathcal{S}^- \#\, \mathcal{S}^+,h_{surg,\varepsilon}^- \# h_{surg,\varepsilon}^+)$ can be continuously deformed back into 
$(S^2 \times (-4,4),h)$ through positive scalar curvature metrics which all coincide 
with $h$ near the ends of $S^2 \times (-4,4)$.
\end{lem}

\begin{figure}
\begin{center}
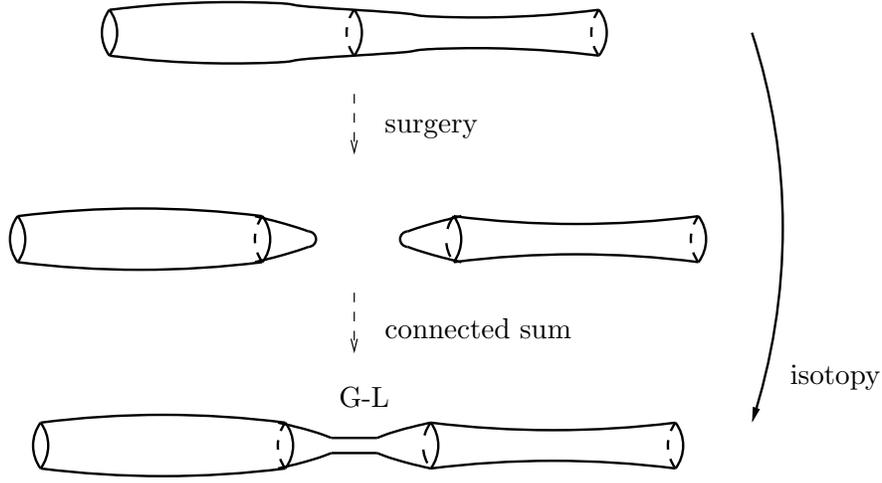
\caption{Reverting surgery.}
\end{center}
\end{figure}

\begin{proof} 
Suppose  $h$ is within $\varepsilon$ of the cylindrical
metric $g_{cyl}=ds^2 + d\theta^2 $ on $S^2 \times (-4,4)$  in the $C^{[1/\varepsilon]}$-topology. It follows from a remark in Section \ref{surgery} and the Proposition \ref{connected.sum.continuity} that the manifold $(\mathcal{S}^- \#\, \mathcal{S}^+,h_{surg,\varepsilon}^- \# h_{surg,\varepsilon}^+)$ can be continuously
deformed into $(\mathcal{S}^- \#\, \mathcal{S}^+,h_{surg,\varepsilon_2}^- \# h_{surg,\varepsilon_2}^+)$ through metrics of positive scalar curvature which all coincide 
with $h$ near the ends of $\mathcal{S}^- \#\, \mathcal{S}^+$.

Let $h_\mu = (1-\eta(s)\mu)h + \eta(s)\mu g_{cyl}$, $\mu \in [0,1]$, where $0 \leq \eta \leq 1$ is a cutoff function such that $\eta(s)$ is 
identically 1
if $|s| \leq 2$ and $\eta$ is identically zero if $|s| \geq 3$. Note that 
$h_\mu = h$ if $|s| \geq 3$ for all $\mu \in [0,1]$, and $h_1 = g_{cyl}$ if $s \in [-2,2]$. If $\varepsilon \leq \varepsilon_3$ with $\varepsilon_3$ sufficiently small, then $h_\mu$ is within $\varepsilon_2$ of the cylindrical
metric $g_{cyl}=ds^2 + d\theta^2 $ in the $C^{[1/\varepsilon_2]}$-topology for every $\mu \in [0,1]$. Therefore the manifold $(\mathcal{S}^- \#\, \mathcal{S}^+,h_{surg,\varepsilon_2}^- \# h_{surg,\varepsilon_2}^+)$ can be continuously deformed into $(\mathcal{S}^- \#\, \mathcal{S}^+,(h_1)_{surg,\varepsilon_2}^- \# (h_1)_{surg,\varepsilon_2}^+)$ through the
positive scalar curvature metrics $(h_\mu)_{surg,\varepsilon_2}^- \# (h_\mu)_{surg,\varepsilon_2}^+$, $\mu \in [0,1]$, which again coincide 
with $h$ near the ends of $\mathcal{S}^- \#\, \mathcal{S}^+$.

Since the metric $h_1$ coincides with  the cylindrical metric $g_{cyl}$ on the region
$s^{-1}([-2,2])$, it follows from Lemma \ref{cylinder.deformation} that the manifold
$$
(\mathcal{S}^- \#\, \mathcal{S}^+,(h_1)_{surg,\varepsilon_2}^- \# (h_1)_{surg,\varepsilon_2}^+)
$$
 can be continuously deformed into $(S^2\times (-4,4),h_1)$
through metrics of positive scalar curvature which equal  $h$ near the ends. The last stage of the deformation is given by  $\mu \in [0,1] \mapsto h_\mu$. 
\end{proof}


\section{Ricci flow with surgery}\label{section.ricci}

In this section we will present some of  the results about the Ricci flow with surgery. We refer the reader to
\cite{CAOZHU06}, \cite{KLEINERLOTT08}, \cite{MORGANTIAN07} and \cite{PERELMAN03A} for more details. Here we will follow more closely the exposition of J. Morgan and G. Tian (compare \cite{MORGANTIAN07}).

Throughout this section $C>0$ and $0<\varepsilon\leq \varepsilon_3$ will be fixed constants.

We will start by defining the several types of  canonical neighborhoods. The  definition of a $(C,\varepsilon)$-cap we give 
here is more restrictive than the one that can be found in the above references. That is due to the fact that we will need more geometric information. The additional assumption we make is in property (4) below: the caps either have positive sectional curvature or are small perturbations of one of two  standard
caps.  We will say more about that  in the end of this section. (Compare with Definition 69.1 of Kleiner and Lott  \cite{KLEINERLOTT08}, property (b)).

A {\it $(C, \varepsilon)$-cap} in a Riemannian manifold $(M^3,g)$ is an open submanifold $\mathcal{C} \subset M$ of positive scalar curvature, together with an open set $N \subset \mathcal{C}$, such that:
\begin{itemize}
 \item [1)] $\mathcal{C}$ is diffeomorphic to an open  3-ball or to  $\re P^3$ minus a ball,
\item[2)] $N$ is an $\varepsilon$-neck with compact complement in $\mathcal{C}$,
\item[3)] $\overline{Y}= \mathcal{C} \setminus N$ is a compact submanifold with boundary. The  interior $Y$ is called the {\it core} of $\mathcal{C}$. The boundary $\d \overline{Y}$ is a central 2-sphere of some $\varepsilon$-neck in $\mathcal{C}$,
\item[4)] $\mathcal{C}$ is of one of the following types:\\ \\ 
{\it type A:} $(\mathcal{C},g)$  has positive sectional curvature everywhere,\\ \\ 
{\it type B:} there exists  a diffeomorphism
$\varphi:(s')^{-1}((-3/\varepsilon, 4+A_0]) \g \mathcal{C} \subset M$ such that the metric $h^{-2}\, \varphi^*(g)$ is within $\varepsilon$ in the $C^{[1/\varepsilon]}$-topology of the  standard initial metric $g_{std}$, where $h=R_g(z)^{-1/2}$ for some $z \in (s')^{-1}(-2/\varepsilon)$,\\ \\
{\it type C:} there exists a double covering $\varphi:S^2 \times (-3/\varepsilon-4,3/\varepsilon+4) \g \mathcal{C}$ with $\varphi(-\theta,-t)=\varphi(\theta,t)$ and such that $h^{-2}\, \varphi^*(g)$ is within $\varepsilon$ of $ds^2+d\theta^2$ in the $C^{[1/\varepsilon]}$-topology, where $h=R_g(z)^{-1/2}$ for some $z \in \varphi(S^2 \times \{-2/\varepsilon-4\})$,
\item[5)] \begin{eqnarray*}
\sup_{x,y \in \mathcal{C}} R(x)/R(y) &<& C,\\
diam \, \mathcal{C} &<& C\, (\sup_{x \in \mathcal{C}} R(x))^{-1/2},\\
vol \, \mathcal{C} &<& C \,(\sup_{x \in \mathcal{C}} R(x))^{-3/2}.
\end{eqnarray*}
\end{itemize}

\begin{figure}
\begin{center}
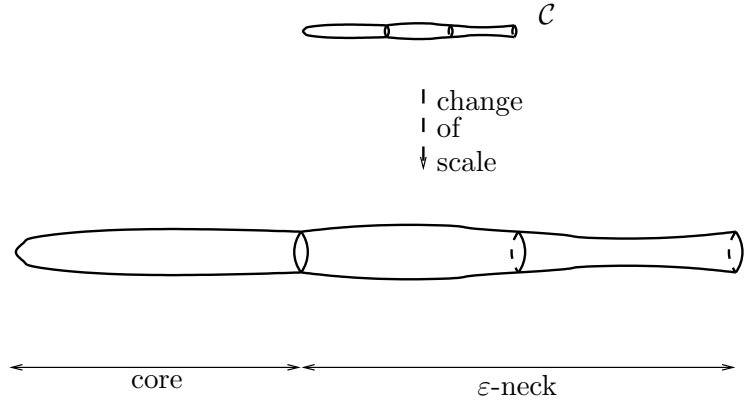
\caption{A $(C,\varepsilon)$-cap after scaling.}
\end{center}
\end{figure}


If a cap $\mathcal{C}$ is of type B, then it is diffeomorphic to a 3-ball and the neck structure 
of $N$ is given by 
$\psi(\theta,t) = \varphi(\theta, t-2/\varepsilon)$. If $\mathcal{C}$ is of type C,  then it is diffeomorphic to $\re P^3$ minus a ball and the neck structure of $N$ is given by $\psi(\theta,t)=\varphi(\theta, t-2/\varepsilon-4)$.

{\bf Remark:} Notice that if a cap $\mathcal{C}$ is of type B, then the metrics $(1-\mu) \varphi^*(g)+\mu\, h^2 \, g_{std}$ have positive scalar curvature for $\mu \in [0,1]$.

A {\it $C$-component} is a compact connected  Riemannian manifold $(M^3,g)$  such that:
\begin{itemize}
 \item [1)] $M$ is diffeomorphic to either $S^3$ or  $\re P^3$,
\item[2)] $(M, g)$ has positive sectional curvature,
\item[3)] $C^{-1}\, (\sup_{x \in M}R(x)) < \inf_\sigma K(\sigma)$, where $\sigma$ varies over all 2-planes of $TM$ and $K(\sigma)$ denotes
the sectional curvature,
\item[4)] $C^{-1} \sup_{x \in M} R(x)^{-1/2}<diam\, M < C \inf_{x \in M} R(x)^{-1/2}$.
\end{itemize}

An {\it $\varepsilon$-round component} is a compact connected Riemannian manifold $(M,g)$ such that there exist
a compact Riemannian manifold $(Z,\overline{g})$ of constant curvature 1, a constant $R>0$, and a diffeomorphism $\varphi: Z \g M$ such that the 
pullback $\varphi^*(Rg)$ is  within $\varepsilon$ of $\overline{g}$  in the $C^{[1/\varepsilon]}$-topology.

Given a Riemannian 3-manifold $(M,g)$, we say that  $x \in M$ has a $(C,\varepsilon)${-\it canonical neighborhood} $U \subset M$ if
one of the following holds:
\begin{itemize}
 \item [1)] $(U,g)$ is an $\varepsilon$-neck centered at $x$,
\item[2)] $(U,g)$ is a $(C,\varepsilon)$-cap whose core contains $x$,
\item[3)] $(U,g)$ is a $C$-component containing $x$,
\item[4)] $(U,g)$ is an $\varepsilon$-round component containing $x$.
\end{itemize}
Notice that  the definition of a $(C,\varepsilon)$-canonical neighborhood is scale invariant.

A Riemannian manifold $(M^3,g_0)$ is said to be {\it normalized}
if 
\begin{itemize}
 \item [1)] $|Rm_{g_0}(x)| \leq 1$ for every $x \in M$,
\item[2)] the volume of any ball of radius one in  $(M^3,g_0)$ is at least $\omega/2$, where $\omega$ denotes
the volume of the unit ball in $\re^3$.
\end{itemize}

The Ricci flow with surgery, with $(M^3,g_0)$ as initial condition, can be thought of as a sequence of standard Ricci flows $(M^3_i,g_i(t))$, each defined for $t \in [t_{i},t_{i+1})$
and becoming singular at $t=t_{i+1}$, where $0=t_0 < t_1 < \cdots < t_{i} < t_{i+1} < \cdots < \infty$ is a discrete set, $M_0=M$, and $g_0(0)=g_0$. The initial
condition $(M^3_{i},g_{i}(t_{i}))$ for each of these Ricci flows is a 
compact Riemannian manifold  obtained from the preceding  Ricci flow $(M^3_{i-1},g_{i-1}(t))_{t \in [t_{i-1},t_{i})}$ by a specific process
called surgery, which depends on some choice of parameters.

If the initial metric has positive scalar curvature, the Ricci flow with surgery becomes extinct at some finite time $T<\infty$. This will mean
that $T=t_{j+1}$ for some $j \geq 0$ and $M_{j+1}=\emptyset$. 

The surgery process depends on some parameters:

\begin{itemize}
\item [1)] the {\it canonical neighborhood parameters} ${\bf r}= r_0 \geq r_1 \geq r_2 \geq \cdots >0$, with
 $r_0=\varepsilon$,
\item [2)] the {\it surgery control parameters} $\Delta = \delta_0 \geq \delta_1 \geq \delta_2 \geq \cdots > 0$, with $\delta_0 \leq \frac16 \varepsilon$  sufficiently small.
\end{itemize}

We say that a Ricci flow $(M^3,g(t))$, $t \in [a,b)$, satisfies {\it the   $(C,\varepsilon)$-canonical neighborhood assumption} with parameter $r$ if every point $(x,t) \in M \times [a,b)$ with $R(x,t) \geq r^{-2}$ has a  $(C,\varepsilon)$-canonical neighborhood.

The Ricci flow with surgery is constructed so that $(M^3_i,g_i(t))_{t \in [t_i,t_{i+1})}$ satisfies the $(C, \varepsilon)$-canonical neighborhood assumption
with parameter $r_i$, for all $0 \leq i \leq j$.

  Let us now describe the surgery at time $t_{i+1}$ with parameters $\delta_i$.  Set $\rho_i = \delta_i \, r_i$, and define
$h = h(\rho_i,\delta_i)$ as in Theorem 11.31 of \cite{MORGANTIAN07}. We have $\rho_i \ll r_i$.

Define
$$
\Omega(t_{i+1}) = \{x \in M: \liminf_{t \g t_{i+1}} R(x,t) < +\infty\}.
$$
If $\Omega(t_{i+1})=\emptyset$, we terminate the flow at time $t_{i+1}$ and declare $M_{i+1}=\emptyset$. Suppose then that $\Omega(t_{i+1})$ is nonempty. It follows from the work of Perelman that the Ricci flow with surgery can be constructed so that (compare Theorem 11.19 in \cite{MORGANTIAN07})  $\Omega(t_{i+1}) \subset M$ is an open set on which  the metrics $g(t)$ converge, in the $C^\infty$ topology
over compact subsets, as $t \g t_{i+1}$, to a metric
$g_i(t_{i+1})$. The scalar curvature $R_{g_i(t_{i+1})}:\Omega(t_{i+1}) \g \re$ is  proper and bounded below.

Let
$$
\Omega_{\rho_i}(t_{i+1}) = \{x \in \Omega(t_{i+1}): R_{g_i(t_{i+1})}(x) \leq \rho_i^{-2}\}.
$$
The set $\Omega_{\rho_i}(t_{i+1}) \subset \Omega(t_{i+1})$ is compact, since $R_{g_i(t_{i+1})}$ is proper. There are finitely many components of $\Omega(t_{i+1})$ which contain points
of $\Omega_{\rho_i}(t_{i+1})$. Denote the union of such components by $\Omega^{big}(t_{i+1})$. Again we terminate the flow if
$\Omega^{big}(t_{i+1})=\emptyset$.

A {\it $2\varepsilon$-horn}  in $(\Omega(t_{i+1}),g_i(t_{i+1}))$ is an open set diffemorphic to $S^2 \times [0,1)$ such that:
\begin{itemize}
\item [1)] the embedding $\psi$ of $S^2 \times [0,1)$ into $\Omega(t_{i+1})$ is a proper map,
\item [2)] every point of the image of this map is the center of a $2\varepsilon$-neck in $\Omega(t_{i+1})$,
\item [3)] the image of the boundary $\psi(S^2 \times \{0\})$ is the central sphere of a $2\varepsilon$-neck in $\Omega(t_{i+1})$.
\end{itemize}

Perelman proved that the open set $\Omega^{big}(t_{i+1})$ contains a finite collection of disjoint $2\varepsilon$-horns $\mathcal{H}_1, \dots, \mathcal{H}_l$, with boundary contained in $\Omega_{\rho_i/2C}(t_{i+1})$, such
that the complement of the union of their interiors is a  compact 3-manifold with boundary which contains $\Omega_{\rho_i}(t_{i+1})$.  For each $1\leq k \leq l$, we can find a strong $\delta_i$-neck
centered at some  $y_k \in \mathcal{H}_k$ with $R_{g_i(t_{i+1})}(y_k)=h^{-2}$, and contained in $\mathcal{H}_k$. Let $S^2_k$ be the central sphere of this neck,  oriented so that the positive $s$-direction 
points toward the end of the horn. Let $\mathcal{H}_k^+$ be the unbounded complementary component of $S^2_k$ inside the horn. The {\it continuing region $C_{t_{i+1}}$ at time $t_{i+1}$} is then defined to be
the complement of $\bigsqcup_{k=1}^l \mathcal{H}_k^+$ in $\Omega^{big}(t_{i+1})$.

We can now do surgery on these $\delta_i$-necks as described in Section 3 of this paper, removing the positive end of the necks and replacing them by small perturbations of standard caps. The result is a compact Riemannian 3-manifold $(M_{i+1},G_{t_{i+1}})$, where  $M_{i+1} = C_{t_{i+1}} \cup_{\sqcup_k S^2_k} B_k$ and each $B_k$ is parametrized by  the ball of
radius $A_0+4$ around the tip of the standard initial metric. The Riemannian metric $G_{t_{i+1}}$ coincides with $g_i(t_{i+1})$ on $C_{t_{i+1}}$, while on each $B_k$ it is a surgery metric like in  Section \ref{surgery}, scaled by $h^{2}$.  Notice that this surgery process removes every component of $\Omega(t_{i+1})$ that does not intersect $\Omega_{\rho_i}(t_{i+1})$. We then say that  $(M_{i+1},G_{t_{i+1}})$ is obtained from the
standard Ricci flow $(M_i^3,g_i(t))_{t \in [t_i,t_{i+1})}$ by {\it doing surgery at the singular time $t_{i+1}$ with parameters $\delta_i$ and $r_i$}. It follows
from the properties of neck surgery that if the metrics $g_i(t)$ have positive scalar curvature, so does $G_{t_{i+1}}$. The Ricci flow $g_{i+1}(t)$ on $M_{i+1}$ is such that $g_{i+1}(t_{i+1})=G_{t_{i+1}}$.

The next result  collects the properties of the Ricci flow with surgery we will need (see Chapter 15 in \cite{MORGANTIAN07}).
\begin{thm}\label{ricci.flow.surgery}
Let $(M^3,g_0)$ be a normalized compact Riemannian manifold, of positive scalar curvature. Suppose there is no $\re P^2$ embedded with trivial normal bundle in $M$.  There exist sequences $r=\{r_i\},\Delta=\{\delta_i\}$, and  a Ricci flow with surgery $(M^3_i,g_i(t))_{t \in [t_i,t_{i+1})}$, $0 \leq i \leq j$, such that: 
\begin{itemize}
 \item [a)] $M_0=M$, and $g_0(0)=g_0$,
\item[b)]   the flow becomes extinct at time $T=t_{j+1}< \infty$,
\item[c)] the  Ricci flow $(M^3_i,g_i(t))_{t \in [t_i,t_{i+1})}$ satisfies the $(C, \varepsilon)$-canonical neighborhood assumption
with parameter $r_i$, for all $0 \leq i \leq j$,
\item[d)] the scalar curvature of $g_i(t)$ is positive, for all $0 \leq i \leq j$ and $t \in [t_i,t_{i+1})$,
\item[e)]  $(M^3_{i+1},g_{i+1}(t_{i+1}))$ is obtained from the  Ricci flow
$(M^3_i,g_i(t))$, $t \in [t_i,t_{i+1})$, by doing surgery at the singular time $t_{i+1}$ with  parameters $\delta_i$ and  $r_i$, for all $0 \leq i \leq j-1$.
\end{itemize}
 \end{thm}
 
 The fact that we can suppose the caps are of one of  3 types - {\it A, B,} or {\it C} - deserves some explanation. This kind of information on the
geometry of the caps is not needed for topological applications - the proof of the Poincaré Conjecture, for example.

 The key point in the 
construction of the Ricci flow with surgery is to prove that a  Ricci flow with surgery defined on an interval $[0,T)$ and such that:
\begin{itemize}
\item [1)] the initial condition is normalized,
\item [2)] the curvature is pinched toward positive,
\item [3)] the canonical neighborhood assumption holds with parameter $r$,
\item [4)] it is $\kappa$-noncollapsed on scales $\leq \varepsilon$,
\end{itemize}
can be extended to an interval $[0,T')$, with $T' >T$ and all four conditions satisfied, perhaps with smaller parameters $r$ and $\kappa$.
 
 The existence of canonical neighborhoods around points of large scalar curvature in the extended flow is established by a contradiction argument,
 after conditions (1), (2), and (4) above are checked. 
If that is not the case,  sequences of Ricci flows with surgery are constructed based at points $p_i$ which violate the canonical neighborhood assumption with parameters $r_i \g 0$. The strong results of Perelman assure that we can take a limit, provided the flows are rescaled so that the
scalar curvature at $p_i$ becomes one. If the limit is a Ricci flow defined all the way back to $-\infty$, then it has to be a $\kappa$-solution. 
In case the limit is only partial, there will be a surgery region near $p_i$ in rescaled distance and time for large $i$. In any case the conclusion is that a neighborhood of $p_i$ either becomes close to a region in a $\kappa$-solution, or it becomes
close to a piece of the standard initial metric $(\re^3,g_{std})$. The standard initial metric  is covered by $\varepsilon$-necks and a $(C,\varepsilon)$-cap of type $B$. The qualitative description  of the $\kappa$-solutions by Perelman (see \cite{PERELMAN03A}) is summarized in Theorem 9.93 of \cite{MORGANTIAN07}.  We have the following possibilities for the
$\kappa$-solution  in the orientable case:
\begin{itemize}
\item [(i)] it is round,
\item [(ii)] it is a $C$-component,
\item [(iii)] it has positive sectional curvature and it is a union of $\varepsilon$-necks and $(C,\varepsilon)$-caps,
\item[(iv)] it is isometric to the cylinder $S^2 \times \re$,
\item[(v)]  it is a quotient of the cylinder $S^2 \times \re$ by the involution $\alpha(\theta,t)=\alpha(-\theta,-t)$. 
\end{itemize}
In all these cases the $\kappa$-solution is a union of canonical neighborhoods. The $(C,\varepsilon)$-caps appear in cases (iii) and 
(v): if (iii) holds then the caps are of type $A$, while if (v) holds the caps are of type $C$. This contradicts the nonexistence 
of canonical neighborhoods around $p_i$, where only caps of types $A$, $B$, or $C$  are considered.


\section{Proof of the Main Results}\label{proofs}

We will start by introducing the concept of a canonical metric.

 Let $h$ be the metric on the unit sphere $S^3$ induced by the standard inclusion $S^3 \subset \re^4$. Given $k,l\geq 0$, let $q_1,\dots,q_k,p_1,\dots,p_l,p_1',\dots,p_l' \in S^3$ be such that
$$
\{q_1,\dots,q_k,p_1,\dots,p_l,p_1',\dots,p_l'\}
$$ 
is a collection of $k+2l$  distinct points in $S^3$. If $\Gamma_1,\dots,\Gamma_k$ are finite subgroups of $SO(4)$ acting freely on $S^3$, let
$q_i' \in S^3/\Gamma_i$ endowed with the quotient metric $g_{\Gamma_i}$ of constant sectional curvature 1  for each $1 \leq i \leq k$. We will
apply the Gromov-Lawson construction to small balls of the same radius centered at $q_1,\dots,q_k,p_1,\dots,p_l,p_1',\dots,p_l'$ in $S^3$ and at $q_i'$ in $S^3/\Gamma_i$. For that we need also to
choose orthonormal bases at each of the points. The
boundaries of the cylindrical necks coming out of $p_j$ and $p_j'$ are identified to each other  with reverse orientations for every $1 \leq j\leq l$, while the same is done to the boundaries of the cylindrical necks coming out of $q_i$ and $q_i'$, for every $1\leq i \leq k$. The resulting manifold is diffeomorphic to a connected sum
$$
M^3 = S^3 \# (S^3/\Gamma_1) \# \dots \# (S^3/\Gamma_k) \# (S^2 \times S^1)
\# \dots \# (S^2 \times S^1),
$$ 
where $l\geq 0$ is the number of $S^2 \times S^1$ summands. Recall that the topology of
such a connected sum sometimes depends on the orientation of the bases chosen at the $q_i'$ (compare Hempel \cite{HEMPEL04}). 
The metric $\hat{g}$ obtained on $M$ from  such construction has positive scalar curvature.  It is also  locally conformally flat. We will refer to metrics isometric to a $\hat{g}$ as above as {\it canonical metrics}. The unit sphere to
which the construction is applied will be called {\it principal sphere}. 

 
 If $\hat{g}_1$ and $\hat{g}_2$ are canonical metrics on the same manifold $M$, then it follows by the uniqueness theorem of Milnor (see \cite{MILNOR62}) that $l_1=l_2$, $k_1=k_2$, and that, after some reordering, there exists an  orientation-preserving
 diffeomorphism  between $S^3/\Gamma_{i,1}$ and $S^3 /\Gamma_{i,2}$ for each $i$.   G. de Rham  proved in \cite{DERHAM50} that in this case there exists an orientation-preserving isometry between  $S^3/\Gamma_{i,1}$ and $S^3 /\Gamma_{i,2}$ for each $i$.
 Recall that the connected sums  $B \#_{\gamma_1} B'$ and $B \#_{\gamma_2} B'$ of constant curvature balls $B$ and $B'$ are isotopic to each other (with the metric unchanged near the ends), if $\gamma_1$ and $\gamma_2$ are planar curves like in Section \ref{connected.sums}. Since we can also move the base points and the orthonormal bases around, according to  Proposition \ref{connected.sum.continuity}, we conclude  that   different canonical metrics on $M$ are isotopic to each other, i.e., live in the same path-connected component of the moduli space
$\mathcal{R}_+(M)/{\rm Diff}(M)$. 


The next result concerns deformations of manifolds which are covered by  canonical neighborhoods.

\begin{prop}\label{canonical.neighborhood}
 Let $(M^3,g)$ be a compact orientable 3-manifold of positive scalar curvature such that  every point $x \in (M^3,g)$ has a 
$(C,\varepsilon)$-canonical neighborhood. Then $g$ is isotopic to  a canonical metric. Moreover, $M^3$ is diffeomorphic to  a space form $S^3/\Gamma$, $\re P^3 \# \re P^3$, or 
$S^2 \times S^1$.
\end{prop}

\begin{proof} If $(M^3,g)$ is a $C$-component, or an $\varepsilon$-round component, the sectional curvatures of $g$ are positive. It follows from 
Hamilton's theorem (see \cite{HAMILTON82}) that the normalized Ricci flow provides a continuous deformation of $g$, through positive scalar curvature
metrics, into a constant curvature space form $S^3/\Gamma$. A canonical metric on $S^3/\Gamma = S^3 \# S^3/\Gamma$, as defined above, is obtained as a Gromov-Lawson connected sum of a round sphere and a round $S^3/\Gamma$.  Since it is 
locally conformally flat, the result follows from Corollary \ref{uniformization} and de Rham's theorem (\cite{DERHAM50}).

Therefore we can assume that every point $x \in M^3$ has a $(C,\varepsilon)$-cap or an $\varepsilon$-neck. The Proposition A.25 of \cite{MORGANTIAN07} implies that $M$ is diffeomorphic to $S^3,\re P^3, \re P^3 \# \re P^3$, or $S^2 \times S^1$.

Suppose $M$ is diffeomorphic to $S^3$.


{\it Claim 1.} There exists a $(C,\varepsilon)$-cap $\mathcal{C} \subset S^3$ with neck $N$ and core $Y$, such that no 
point of $s_{N}^{-1}(0.9/\varepsilon)$ is in the core of a $(C,\varepsilon)$-cap that contains $\mathcal{C}$. The neck $N$ is oriented so that the positive $s$-direction points towards the boundary of the cap.

 It follows from the results in the appendix of \cite{MORGANTIAN07} that if every point $x \in (M^3,g)$ has an $\varepsilon$-neck, then $M$ is diffeomorphic to  $S^2 \times S^1$. Therefore $(S^3,g)$  must contain a $(C,\varepsilon)$-cap $\mathcal{C}_1$.  

Suppose the claim is false. Then there exists an  infinite chain of $(C,\varepsilon)$-caps $\mathcal{C}_1 \subset \mathcal{C}_2 \subset \cdots \subset$ in $S^3$, of necks $N_1, N_2, \dots$ and cores $Y_1, Y_2, \dots$, 
such that $s_{N_i}^{-1}(0.9/\varepsilon) \cap Y_{i+1} \ne \emptyset$ for every $i \geq 1$. Let $\varphi_i:S^2 \times (-1/\varepsilon,1/\varepsilon) \g N_i$ be the $\varepsilon$-neck structure on
$N_i = \mathcal{C}_i-\overline{Y_i}$,  oriented so that the positive $s$-direction points towards the boundary of the cap. 

  Denote by  $\mathcal{C}_i^{0.9}$ and $\mathcal{C}_i^0$  the connected components of
$\mathcal{C}_i - s_{N_i}^{-1}(0.9/\varepsilon)$ and $\mathcal{C}_i - s_{N_i}^{-1}(0)$ that contain the core $Y_i$, respectively. Let us prove that $\d\mathcal{C}_i^{0.9} \subset \mathcal{C}_{i+1}^0$. There are two cases to consider: $\d\mathcal{C}_i^{0.9} \cap N_{i+1}=\emptyset$ and 
$\d\mathcal{C}_i^{0.9} \cap N_{i+1} \ne \emptyset$. If $\d\mathcal{C}_i^{0.9} \cap N_{i+1}=\emptyset$, then $\d\mathcal{C}_i^{0.9} \subset \overline{Y_{i+1}} \subset \mathcal{C}_{i+1}^0$.  If $\d\mathcal{C}_i^{0.9} \cap N_{i+1} \ne \emptyset$, then $N_i \cap N_{i+1} \ne \emptyset$ and it follows from Proposition A.11 of \cite{MORGANTIAN07} that $h_{N_i} \leq 1.1 h_{N_{i+1}}.$ It follows from the definition of an $\varepsilon$-neck that the diameter of the sphere $\d\mathcal{C}_i^{0.9}$ is at most  $2\pi h_{N_i}$. On the other hand any curve connecting a point of $Y_{i+1}$ to $\d \mathcal{C}_{i+1}^0$ must cross the left-hand half of $N_{i+1}$. This implies that the distance of $Y_{i+1}$ to $\d \mathcal{C}_{i+1}^0$ is bounded below by $0.9\varepsilon^{-1}h_{N_{i+1}}$.  We have concluded that (by choosing $\varepsilon$ sufficiently small)
$$
diam(\d\mathcal{C}_i^{0.9}) < d(Y_{i+1},\d \mathcal{C}_{i+1}^0).
$$
Since $\d \mathcal{C}_i^{0.9} \cap Y_{i+1} \ne \emptyset$ and $\d \mathcal{C}_i^{0.9}$ is connected, we get that $\d\mathcal{C}_i^{0.9} \subset \mathcal{C}_{i+1}^0$. In any case we have $\d\mathcal{C}_i^{0.9} \subset \mathcal{C}_{i+1}^0$. 

If it also holds that $\d \mathcal{C}_{i+1}^0 \subset \mathcal{C}_i^{0.9}$, we get that the set $\mathcal{C}_i^{0.9} \cup \mathcal{C}_{i+1}^0$ is both open and closed in $S^3$. Since $S^3$ is connected we conclude that  
$S^3= \mathcal{C}_i^{0.9} \cup \mathcal{C}_{i+1}^0\subset \mathcal{C}_{i+1}$ and hence $S^3=\mathcal{C}_{i+1}$. This is not possible since 
$S^3$ is compact. We have already proved that $\d\mathcal{C}_i^{0.9} \subset \mathcal{C}_{i+1}^0$, hence  $\d \mathcal{C}_{i+1}^0\cap \d \mathcal{C}_i^{0.9} =  \emptyset$. Therefore 
$\d \mathcal{C}_{i+1}^0 \subset S^3-\mathcal{C}_i^{0.9}$, since $\d \mathcal{C}_{i+1}^0$ is connected. Since $\mathcal{C}_i^{0.9}$ is connected, this implies that
 $\mathcal{C}_i^{0.9} \subset \mathcal{C}_{i+1}^0$, and hence the regions $s_{N_i}^{-1}((0.25/\varepsilon,0.75/\varepsilon))$ are 
pairwise disjoint.  Since there are infinitely many of them and the  scalar curvature of $S^3$ is bounded below by a positive
constant, this
contradicts the finiteness of volume. This ends the proof of the claim.

{\it Claim 2.}  There exists a structured chain of $\varepsilon$-necks $\{N_1'=N,N_2',\dots,N_a'\}$ and a $(C,\varepsilon)$-cap 
$\tilde{\mathcal{C}}$ with neck $\tilde{N}$ and core $\tilde{Y}$ such that:
\begin{itemize}
 \item [1)] $N_i' \cap Y = \emptyset$ for all $1\leq  i \leq a$,
\item[2)] $s_{N_a'}^{-1}(0.9/\varepsilon) \cap \tilde{Y} \neq \emptyset$,
\item[3)] $s_{\tilde{N}}^{-1}(0) \subset \overline{Y} \cup s_{N_1'}^{-1}((-1/\varepsilon,0.9/\varepsilon)) \cup N_2' \cup \dots \cup N_{a-1}'\cup s_{N_a'}^{-1}((-1/\varepsilon,0.9/\varepsilon))$,
\item[4)] $S^3 = \mathcal{C} \cup N_2' \cup \dots \cup N_a' \cup \tilde{C}$.
\end{itemize}

\begin{figure}
\begin{center}
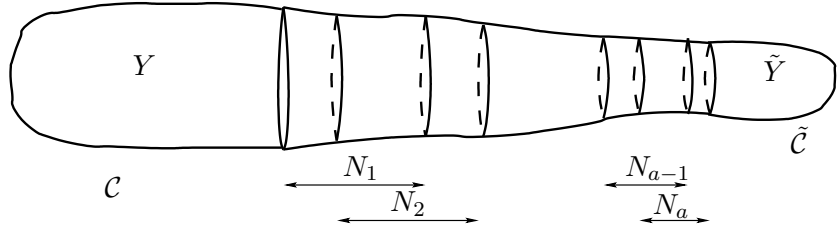
\caption{A covering of $(S^3,g)$ by $\varepsilon$-necks and $\varepsilon$-caps.}
\end{center}
\end{figure}

Let $\{N_1'=N,N_2',\dots,N_j'\}$ be a structured chain of $j$ $\varepsilon$-necks, $j \geq 1$, such that 
$N_i' \cap Y = \emptyset$ for all $1\leq  i \leq j$. Choose $z \in s_{N_j'}^{-1}(0.9/\varepsilon)$. Then either $z$ is the center
of an $\varepsilon$-neck $N_{j+1}'$, or $z$ is contained in the core $\tilde{Y}$ of a $(C,\varepsilon)$-cap $\tilde{\mathcal{C}}$
of neck $\tilde{N}$. If $z$ is the center of an $\varepsilon$-neck $N_{j+1}'$, and since
the 2-sphere $s_{N_j'}^{-1}(0)$ separates $S^3$, we have $N_{j+1}' \cap Y = \emptyset$. We have obtained a structured chain of $j+1$ 
$\varepsilon$-necks   $\{N_1'=N,N_2',\dots,N_{j+1}'\}$ with $N_i' \cap Y = \emptyset$ for all $1\leq  i \leq j+1$. 

Let us prove that there can be no infinite structured chain of $\varepsilon$-necks  $\{N_1',N_2',\dots,\}$ in $S^3$. If 
$\{N_1',N_2',\dots,N_a'\}$ is a structured chain of $\varepsilon$-necks, then the sets $s_{N_i'}^{-1}(-0.25/\varepsilon,0.25/\varepsilon)$ are mutually disjoint. Since the scalar curvature of $(S^3,g)$ is bounded 
below by a positive constant, an infinite number of these sets would contradict the finiteness of the volume. This proves the assertion.

It follows that  there must exist a structured chain of $\varepsilon$-necks $\{N_1'=N,N_2',\dots,N_j'\}$, $j \geq 1$, with 
$N_i' \cap Y = \emptyset$ for all $1\leq  i \leq j$, and such that every $z \in s_{N_j'}^{-1}(0.9/\varepsilon)$ is contained in the core $\tilde{Y}$ of a $(C,\varepsilon)$-cap $\tilde{\mathcal{C}}$
of neck $\tilde{N}$. 

Let us fix $z$ and $\tilde{\mathcal{C}}$ as above.
If $s_{N_j'}^{-1}(0.9/\varepsilon) \cap s_{\tilde{N}}^{-1}(0) \ne \emptyset$, it would follow by estimating distances, as in the proof of Claim 1, that $s_{N_j'}^{-1}(0.9/\varepsilon) \subset \tilde{N}$. This cannot be true since $s_{N_j'}^{-1}(0.9/\varepsilon) \cap \tilde{Y} \neq \emptyset$. Hence 
$s_{N_j'}^{-1}(0.9/\varepsilon) \cap s_{\tilde{N}}^{-1}(0) = \emptyset$. Notice that the region
$$
R = \overline{Y} \cup s_{N_1'}^{-1}((-1/\varepsilon,0.9/\varepsilon)) \cup N_2' \cup \dots \cup N_{j-1}'\cup s_{N_j'}^{-1}((-1/\varepsilon,0.9/\varepsilon)),
$$
is connected and $\d R = s_{N_j'}^{-1}(0.9/\varepsilon)$. Therefore either $s_{\tilde{N}}^{-1}(0) \subset S^3 - \overline{R}$, or
$s_{\tilde{N}}^{-1}(0) \subset R$. Let $\tilde{\mathcal{C}}^0$ be the connected component of $\tilde{\mathcal{C}} - s_{\tilde{N}}^{-1}(0)$ that contains the core $\tilde{Y}$. 

If $s_{\tilde{N}}^{-1}(0) \subset S^3 - \overline{R}$, then the connectedness of $R$ implies
 $R \subset \tilde{\mathcal{C}}^0$. Since the diameter of 
 $s_{N_j'}^{-1}([0.9/\varepsilon,1/\varepsilon))$ is at most $0.2 h_{N_j'}/\varepsilon$ and 
 $d(\tilde{\mathcal{C}}^0,\d \tilde{\mathcal{C}})>0.9h_{\tilde{N}}/\varepsilon$, it follows by distance comparison as before that 
 $s_{N_j'}^{-1}([0.9/\varepsilon,1/\varepsilon)) \subset \tilde{\mathcal{C}}$.
 This implies $\mathcal{C} \subset \tilde{\mathcal{C}}$.  If $j=1$ this is already in contradiction with 
the choice of $\mathcal{C}$ made in the
previous claim, since in this case we would have $z \in s_{N}^{-1}(0.9/\varepsilon) \cap \tilde{Y}$. If $j \geq 2$, we notice that $z \in s_{N_j'}^{-1}(0.9/\varepsilon) \cap \tilde{Y}$ and the points of 
$s_{\tilde{N}}^{-1}(0) \subset S^3-\overline{R}$ are in the same connected component of the complement in $S^3$ of the 2-sphere $s_{N}^{-1}(0.9/\varepsilon)$.  Therefore we cannot have $s_{N}^{-1}(0.9/\varepsilon) \subset \tilde{N} \cap \tilde{\mathcal{C}}^0$ (by item 4 of Prop. A.11 in \cite{MORGANTIAN07}, this would imply that $s_{N}^{-1}(0.9/\varepsilon)$ separates $s_{\tilde{N}}^{-1}(0)$ from $\tilde{Y}$). We conclude that
$s_{N}^{-1}(0.9/\varepsilon) \cap \tilde{Y} \ne \emptyset$, and again this is in contradiction with Claim 1. 

Hence $\d \tilde{\mathcal{C}}^0=s_{\tilde{N}}^{-1}(0) \subset R$. Since $\d R \cap \tilde{Y} \neq \emptyset$, it follows (by distance estimates) that $\d R \subset \tilde{\mathcal{C}}^0$. Therefore it follows by connectedness of $S^3$ that $S^3 = R \cup \tilde{\mathcal{C}}^0$. This finishes the
proof of the claim.

\begin{figure}
\begin{center}
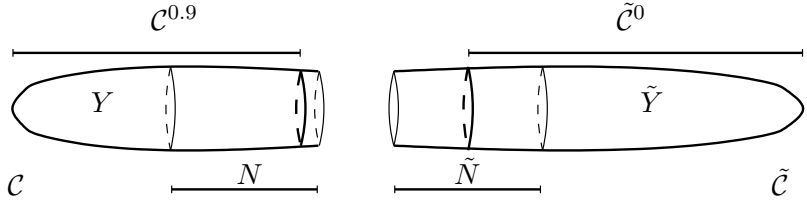
\caption{The $\varepsilon$-caps $\mathcal{C}$ and $\tilde{\mathcal{C}}$.}
\end{center}
\end{figure}

We choose to orient the neck $\tilde{N}$ differently, so that the positive $s$-direction points towards the core $\tilde{Y}$. Hence $\tilde{\mathcal{C}}^0 = \overline{\tilde{Y}} \cup s_{\tilde{N}}^{-1}((0,1/\varepsilon))$. 

Let 
$\mathcal{C}^{0.9} = \overline{Y} \cup s_{N}^{-1}((-1/\varepsilon,0.9/\varepsilon))$. 

Since the manifold is diffeomorphic to the 3-sphere, only caps of types $A$ and $B$ can appear. We will divide the rest
of the proof in four cases according with the types of the caps.

{\bf Case I.} $\mathcal{C}$ and $\tilde{\mathcal{C}}$ are of type $A$.  

Suppose $\d \tilde{\mathcal{C}}^0 \subset \mathcal{C}^{0.9}$.  Then either
 $\d \mathcal{C}^{0.9} \subset int\, (S^3 - \tilde{\mathcal{C}}^0)$ or $\d \mathcal{C}^{0.9} \subset \tilde{\mathcal{C}}^0$. If
$\d \mathcal{C}^{0.9} \subset int\, (S^3 - \tilde{\mathcal{C}}^0)$, then $ \tilde{\mathcal{C}}^0 \subset \mathcal{C}^{0.9}$. This is because the closure of $\tilde{\mathcal{C}}^0$ is connected, and therefore it cannot intersect both $\mathcal{C}^{0.9}$ and the complement of  $\mathcal{C}^{0.9}$ without intersecting the boundary
$\d \mathcal{C}^{0.9}$. This is in contradiction with property (2) of the previous claim, since
$\tilde{Y} \subset  \tilde{\mathcal{C}}^0 $. Therefore $\d \mathcal{C}^{0.9} \subset \tilde{\mathcal{C}}^0$, and hence $S^3 = \tilde{\mathcal{C}}^0  \cup \mathcal{C}^{0.9}$ since in that case we would have $\tilde{\mathcal{C}}^0  \cup \mathcal{C}^{0.9}$ both open and closed in $S^3$. This implies that $(S^3,g)$ has 
positive sectional curvature, and the result follows from \cite{HAMILTON82}.

If $\d \tilde{\mathcal{C}}^0$ is not contained in  $\mathcal{C}^{0.9}$, then there must exist $2\leq i \leq a$ such that
$\d \tilde{\mathcal{C}}^0 \subset s_{N_i'}^{-1}((-0.01/\varepsilon,0.95/\varepsilon))$. This implies that the distance of any point in $s_{N_i'}^{-1}((-4,4))$ to $s_{\tilde{N}}^{-1}(0)$ is strictly less than $d(s_{\tilde{N}}^{-1}(0), \d \tilde{\mathcal{C}})$. Hence  $s_{N_i'}^{-1}((-4,4))\subset \tilde{\mathcal{C}}$. In particular we obtain that $s_{N_i'}^{-1}((-4,4))$ and the component bounded by $s_{N_i'}^{-1}(0)$ in $\tilde{\mathcal{C}}$ have positive sectional curvature. It follows from the Interpolation Lemmas 
\ref{interpolation.lemma.1} and \ref{interpolation.lemma.2} that there exists a diffeomorphism $\psi:S^2 \times (-1/\varepsilon,\beta) \g \bigcup_{j=1}^i N_j'$, such that:
\begin{itemize}
 \item [1)] $\psi\,(\theta,t) = \psi_1(\theta,t)$ for $(\theta,t) \in S^2 \times (-1/\varepsilon,0.25/\varepsilon)$,
\item [2)] $\psi\, (\theta,t) = \psi_i(A\cdot \theta,t-\beta+1/\varepsilon)$ for $(\theta, t) \in S^2 \times (\beta-1.25/\varepsilon, \beta)$, where $A$ is an isometry of $(S^2,d\theta^2)$,
\item[3)] there exists  a continuous path of metrics $\mu \in [0,1] \mapsto g_\mu$ of positive scalar curvature
on $S^2 \times (-1/\varepsilon, \beta)$, with
$g_0 = \psi^*(g)$ and  $g_1$ rotationally symmetric, and such that it restricts to the linear homotopy
$g_\mu = (1-\mu)\psi^*(g) + \mu \, h_1^2 g_{cyl}$
on $S^2 \times (-1/\varepsilon,0.25/\varepsilon)$ and  to the linear homotopy
$g_\mu = (1-\mu)\psi^*(g) + \mu \, h_i^2 g_{cyl}$  on $S^2 \times (\beta-1.25/\varepsilon, \beta)$.
\end{itemize}
Here $\psi_j : S^2 \times (-1/\varepsilon,1/\varepsilon) \g N_j'$ denotes the $\varepsilon$-neck structure associated to the neck $N_j'$.

We perform surgery along the central spheres $S_1=s_{N_1'}^{-1}(0)$ and $S_i=s_{N_i'}^{-1}(0)$, and glue standard caps to both left and right sides of each sphere as explained in Section \ref{surgery}. (Here we could have performed surgery along $S_1$ and $S_a=s_{N_a'}^{-1}(0)$ as in Section \ref{outline}, but we choose to do it along $S_1$ and $S_i$ so that the proof can be more easily modified to handle the other cases). In doing this we break the manifold into three 
components: $(\mathcal{S}_1,g_1),(P,g_P)$, and $(\mathcal{S}_2,g_2)$. It follows from Corollary \ref{positive.curvature.surgery} that the left-hand
$(\mathcal{S}_1,g_1)$ and the right-hand $(\mathcal{S}_2,g_2)$ components have positive sectional curvature, since the same holds for the caps $\mathcal{C}$ and $\tilde{\mathcal{C}}$. Therefore they can be deformed to  constant curvature metrics by the normalized Ricci flow.  The middle component $(P,g_P)$  is obtained by attaching standard caps
to the boundary of the region between the spheres $S_1$ and $S_i$.   Since the deformation of item (3) above restricts to linear homotopies
on neighborhoods of the  surgery spheres, it follows from Lemma \ref{deforming.surgery}  that it can be extended to the caps. This provides a
 deformation of the metric  on $P$, through metrics of positive scalar curvature, that ends in a rotationally symmetric manifold. 
It follows from Corollary \ref{uniformization} that this locally conformally flat manifold can be deformed, through metrics of positive scalar curvature, 
into one of constant sectional curvature.  

We have proved that there exist continuous families of positive scalar curvature metrics  $g_{1,\mu}, g_{P,\mu}, g_{2,\mu}$ on
$\mathcal{S}_1, P$, $\mathcal{S}_2$, respectively, $\mu \in [0,1]$, such that $g_{1,0}=g_1$, $g_{P,0}=g_P$, $g_{2,0}=g_2$, and
so that $g_{1,1}, g_{P,1}$, and $g_{2,1}$ are round. It follows from Proposition \ref{connected.sum.continuity} that there is
a choice of parameters so that we can consider the continuous family of connected sums at the surgery tips: 
$$
g^{\#}_\mu= (g_{1,\mu} \,\#\, g_{P,\mu} \, \#\, g_{2,\mu})_{\mu \in [0,1]} {\rm \,\, on \,\, } \mathcal{S}_1 \,\#\, P \,\#\, \mathcal{S}_2. 
$$

The Proposition \ref{undoing.surgery} implies that $(S^3,g)$ can be continuously deformed, through positive scalar curvature metrics,
into $(\mathcal{S}_1 \,\#\, P \,\#\, \mathcal{S}_2, g_0^{\#})$. The metric $g_1^\#$ is a Gromov-Lawson connected sum of three round spheres, hence it is locally conformally flat by a remark in Section \ref{connected.sums}. The proof of Case I finishes with Corollary \ref{uniformization}.

\begin{figure}
\begin{center}
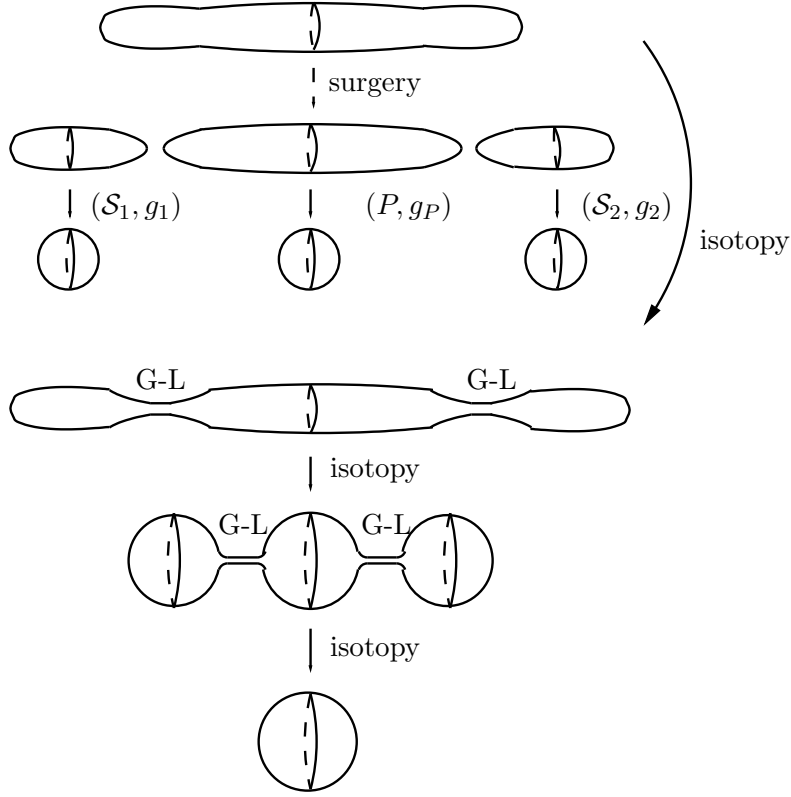
\caption{Illustrating Case I.}
\end{center}
\end{figure}

{\bf Case II} $\mathcal{C}$ is of type $A$, and $\tilde{\mathcal{C}}$ is of type $B$.

Suppose $\d \tilde{\mathcal{C}}^0  \subset \mathcal{C}^{0.9}$.  Then  $\d \mathcal{C}^{0.9} \subset \tilde{\mathcal{C}}^0$ and $S^3 = \tilde{\mathcal{C}}^0 \cup \mathcal{C}^{0.9}$, as in the proof of Case I. We can  perform surgery along the central sphere   $\tilde{S}=s_{\tilde{N}}^{-1}(0)$, and glue standard caps to both left and right sides of it. In doing this we break the manifold into two
components: $(\mathcal{S}_1,g_1)$ and $(\mathcal{S}_2,g_2)$. Since $s^{-1}_{\tilde{N}}((-4,4)) \subset \mathcal{C}$, it follows from Corollary \ref{positive.curvature.surgery} that the left-hand
$(\mathcal{S}_1,g_1)$ component  has positive sectional curvature. Therefore it can be deformed to a  constant curvature metric by the normalized Ricci flow.  It follows from Lemma \ref{deforming.surgery} and  a remark of Section \ref{section.ricci} that the right-hand  component $(\mathcal{S}_2,g_2)$  can be deformed, through metrics of positive scalar curvature, into a rotationally symmetric manifold. (Recall that the standard initial metric
is rotationally symmetric). It follows from Corollary \ref{uniformization} that this manifold can be deformed 
into one of constant sectional curvature.  Since both of the components $(\mathcal{S}_1,g_1)$ and  $(\mathcal{S}_2,g_2)$ can be
deformed to round metrics, the original manifold $(S^3,g)$ is isotopic to a Gromov-Lawson connected sum of round spheres (two of them) as in
the proof of Case I  and the
result follows from Corollary \ref{uniformization}.  
 
If $\d \tilde{\mathcal{C}}^0$ is not contained in  $\mathcal{C}^{0.9}$, then there must exist $2\leq i \leq a$ such that
$\d \tilde{\mathcal{C}}^0 \subset s_{N_i'}^{-1}((-0.01/\varepsilon,0.95/\varepsilon))$. It follows from the Interpolation Lemmas 
\ref{interpolation.lemma.1} and \ref{interpolation.lemma.2} that there exists a diffeomorphism 
$$
\psi:S^2 \times (-1/\varepsilon,\beta) \g \bigcup_{j=1}^i N_j' \cup \tilde{N},
$$
such that:
\begin{itemize}
 \item [1)] $\psi\,(\theta,t) = \psi_1(\theta,t)$ for $(\theta,t) \in S^2 \times (-1/\varepsilon,0.25/\varepsilon)$,
\item [2)] $\psi\, (\theta,t) = \psi_{\tilde{N}}(A\cdot \theta,t-\beta+1/\varepsilon)$ for $(\theta, t) \in S^2 \times (\beta-0.75/\varepsilon, \beta)$, where $A$ is an isometry of $(S^2,d\theta^2)$,
\item[3)] there exists  a continuous path of metrics $\mu \in [0,1] \mapsto g_\mu$ of positive scalar curvature
on $S^2 \times (-1/\varepsilon, \beta)$, with
$g_0 = \psi^*(g)$ and  $g_1$ rotationally symmetric, and such that it restricts to the linear homotopy
$g_\mu = (1-\mu)\psi^*(g) + \mu \, h_1^2 g_{cyl}$
on $S^2 \times (-1/\varepsilon,0.25/\varepsilon)$ and  to the linear homotopy
$g_\mu = (1-\mu)\psi^*(g) + \mu \, \tilde{h}^2 g_{cyl}$  on $S^2 \times (\beta-0.75/\varepsilon, \beta)$.
\end{itemize}

We can  perform surgery along the central sphere $S_1=s_{N}^{-1}(0)$, and glue standard caps to both left and right sides of it.
 In doing this we break the manifold into two 
components: $(\mathcal{S}_1,g_1)$, and $(\mathcal{S}_2,g_2)$. It follows from Corollary \ref{positive.curvature.surgery} that the left-hand
$(\mathcal{S}_1,g_1)$ has positive sectional curvature. Therefore it can be deformed to  a constant curvature metric by the normalized Ricci flow.   Since the deformation of item (3) above restricts to linear homotopies
on the regions $\psi_1(S^2 \times (-1/\varepsilon,0.25/\varepsilon))$ and $\psi_{\tilde{N}}(S^2 \times (0.25/\varepsilon, 1/\varepsilon))$ , it follows from Lemma \ref{deforming.surgery} and a remark of Section \ref{section.ricci} that it can be extended to the caps. This provides a
 deformation of the metric  on $\mathcal{S}_2$, through metrics of positive scalar curvature, that ends in a 
rotationally symmetric manifold. The proof of Case II proceeds similarly as before, by using Corollary \ref{uniformization} and making Gromov-Lawson connected sums.

{\bf Case III} $\mathcal{C}$ is of type $B$, and  $\tilde{\mathcal{C}}$ is of type $A$.

Suppose $\d \tilde{\mathcal{C}}^0 \subset \mathcal{C}^{0.9}$.  Then  $\d \mathcal{C}^{0.9} \subset \tilde{\mathcal{C}}^0$ and $S^3 = \tilde{\mathcal{C}}^0  \cup \mathcal{C}^{0.9}$, as in the proof of Case I.  Notice that this also implies that the distance of any point of $s_{N}^{-1}((-4,4))$ to $\tilde{\mathcal{C}}^0$ is strictly less than $d(\tilde{\mathcal{C}}^0, \d \tilde{\mathcal{C}})$. Therefore $s_{N}^{-1}((-4,4)) \subset \tilde{\mathcal{C}}$. In particular we obtain that $s_{N}^{-1}(-4,4)$ and the component bounded by $s_{N}^{-1}(0)$ in $\tilde{\mathcal{C}}$ have positive sectional curvature.
We can  perform surgery along the central sphere   $S=s_{N}^{-1}(0)$, and glue standard caps to both left and right sides of it. In doing this we break the manifold into two
components: $(\mathcal{S}_1,g_1)$ and $(\mathcal{S}_2,g_2)$.  It follows from Corollary \ref{positive.curvature.surgery} that the right-hand
$(\mathcal{S}_2,g_2)$ component  has positive sectional curvature, and therefore can be deformed to a round sphere by the normalized Ricci flow. The left-hand component $(\mathcal{S}_1,g_1)$  can be deformed into a rotationally symmetric manifold with the same argument used for the right-hand component in Case II. The proof proceeds similarly as before. 

If $\d \tilde{\mathcal{C}}^0$ is not contained in  $\mathcal{C}^{0.9}$,  the result follows from the previous arguments by interpolating the intermediate necks.

{\bf Case IV} $\mathcal{C}$ and  $\tilde{\mathcal{C}}$ are  of type $B$.

If $\d \tilde{\mathcal{C}}^0$ is not contained in  $\mathcal{C}^{0.9}$, then there must exist $2\leq i \leq a$ such that
$\d \tilde{\mathcal{C}}^0 \subset s_{N_i'}^{-1}((-0.01/\varepsilon,0.95/\varepsilon))$. It follows from the Interpolation Lemmas 
\ref{interpolation.lemma.1} and \ref{interpolation.lemma.2} that there exists a diffeomorphism 
$$
\psi:S^2 \times (-1/\varepsilon,\beta) \g \bigcup_{j=1}^i N_j' \cup \tilde{N},
$$
with the same properties as in the proof of Case II.

Since the deformation of item (3) above restricts to linear homotopies
on the regions $\psi_1(S^2 \times (-1/\varepsilon,0.25/\varepsilon))$ and $\psi_{\tilde{N}}(S^2 \times (0.25/\varepsilon, 1/\varepsilon))$, it follows from a remark in Section \ref{section.ricci} that it extends as linear homotopies to the caps. Therefore the entire
manifold $(S^3,g)$  can be deformed, through metrics of positive scalar curvature, into a rotationally symmetric manifold. The result
follows from Corollary \ref{uniformization}. 

Suppose $\d \tilde{\mathcal{C}}^0\subset \mathcal{C}^{0.9}$. Then  $\d \mathcal{C}^{0.9} \subset \tilde{\mathcal{C}}^0$ and $S^3 = \tilde{\mathcal{C}}^0  \cup \mathcal{C}^{0.9}$, as in the proof of Case I. Let $\mathcal{C}^0 = \overline{Y} \cup s_{N}^{-1}((-1/\varepsilon,0))$. Since $\d \mathcal{C}^{0.9} \subset \tilde{\mathcal{C}}^0$, it follows by distance comparison that
$\d \mathcal{C}^0 \subset s_{\tilde{N}}^{-1}((-0.95/\varepsilon, 0)) \cup \tilde{\mathcal{C}}^0$. Let
$\psi_\mathcal{C}:S^2 \times (-2/\varepsilon,1/\varepsilon) \g \mathcal{C}$ and $\psi_{\tilde{\mathcal{C}}}:S^2 \times (-1/\varepsilon,2/\varepsilon) \g \tilde{\mathcal{C}}$ be the extensions of $\psi_1$ and $\psi_{\tilde{N}}$ given
by the standard cap structures of $\mathcal{C}$ and $\tilde{\mathcal{C}}$, respectively.

{\it Claim.} $\psi_\mathcal{C}((-1.5/\varepsilon,0.5/\varepsilon)) \cap \psi_{\tilde{\mathcal{C}}}(-0.5/\varepsilon,1.5/\varepsilon)) 
\neq \emptyset$.

Suppose $\tilde{h} \leq h_N$. If $\psi_\mathcal{C}((-1.5/\varepsilon,0.5/\varepsilon)) \cap \psi_{\tilde{\mathcal{C}}}(-0.5/\varepsilon,1.5/\varepsilon)) 
= \emptyset$, it follows from the connectedness of the regions and the inclusion $\d \tilde{\mathcal{C}}^0 \subset \mathcal{C}^{0.9}$
that $\psi_{\tilde{\mathcal{C}}}((-0.5/\varepsilon,1.5/\varepsilon))$ is contained in the component of 
$\mathcal{C} - \psi_\mathcal{C}(S^2 \times \{-1.5/\varepsilon\})$ disjoint from $s_{N}^{-1}(0)$.
Since the distance of any point in $\tilde{\mathcal{C}}$ to the central sphere $s_{\tilde{N}}^{-1}(0)$ is at most $(2.1)\tilde{h}/\varepsilon$, and the distance from any point of $\psi_\mathcal{C}(S^2 \times \{-1.5/\varepsilon\})$ to
$\d \mathcal{C}$ is at least $(2.25)h_N/\varepsilon$, we obtain $\tilde{\mathcal{C}} \subset \mathcal{C}$. This implies
$\mathcal{C}=S^3$, which is a contradiction. If $h_N \leq \tilde{h}$, the proof is similar (we would show that $\tilde{\mathcal{C}}=S^3$). This finishes the proof of the claim.

Since $\psi_\mathcal{C}((-1.5/\varepsilon,0.5/\varepsilon)) \cap \psi_{\tilde{\mathcal{C}}}(-0.5/\varepsilon,1.5/\varepsilon)) 
\neq \emptyset$, it follows from Lemma \ref{interpolation.lemma.1} and  a remark in Section \ref{section.ricci}  that the manifold $(S^3,g)$ can be deformed, through metrics of positive scalar curvature, into a rotationally symmetric manifold. The result
follows from Corollary \ref{uniformization}. 


We have finished the case of $S^3$. The proof is similar when $M$ is diffeomorphic to $\re P^3$ or $\re P^3 \# \re P^3$. The arguments used to handle caps of type $B$ can be easily modified to apply for caps of type $C$. The key property is that these caps are $\varepsilon$-close to locally conformally flat metrics of positive scalar curvature which are cylindrical near the end.

Suppose now that $M$ is diffeomorphic to $S^2 \times S^1$. In that case it follows from the proof of Proposition A.21 in \cite{MORGANTIAN07} that every point is contained in an $\varepsilon$-neck whose central sphere does not separate $M$.
Let $N$ be one such neck, with central sphere $S$. Do surgery on $N$ along $S$ and glue standard caps to both sides of it. The resulting manifold is a 3-sphere $(S^3,g_{surg})$ endowed with a metric of positive scalar curvature such that every point of it has a canonical neighborhood. The previous arguments imply that  $(S^3,g_{surg})$ can be deformed into a round sphere. It follows from Lemma \ref{undoing.surgery} that the original manifold $(S^2 \times S^1,g)$ is isotopic to the Gromov-Lawson connected sum of 
$(S^3,g_{surg})$ with itself, where the connected sum is performed at the tips of the spherical caps. The result follows  since
a canonical metric on $S^2 \times S^1$  is defined as a Gromov-Lawson connected sum of a round 3-sphere to itself.

\begin{figure}
\begin{center}
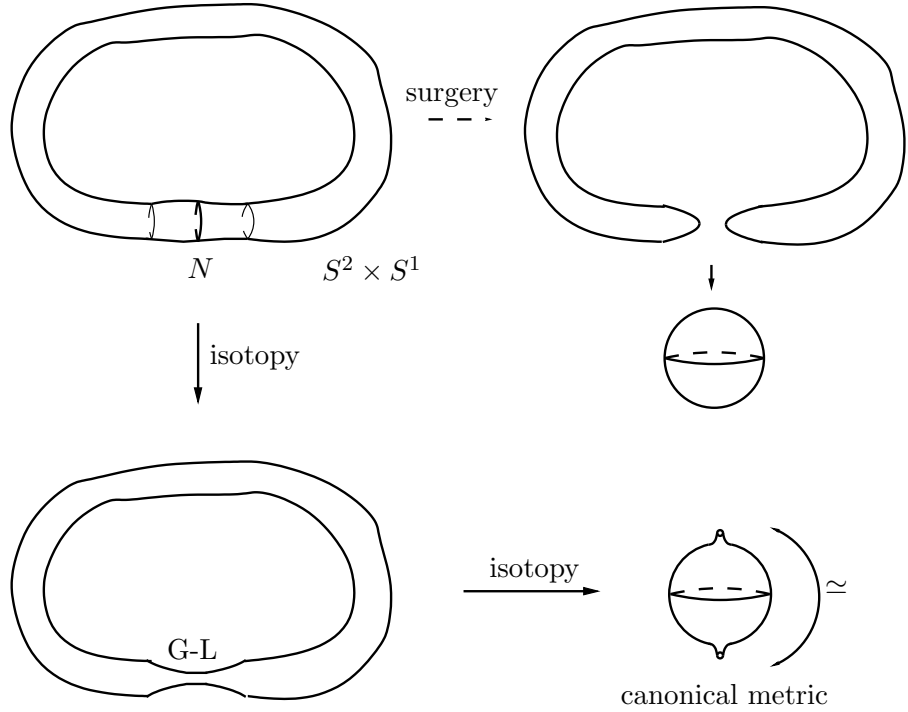
\caption{Deforming a union of $\varepsilon$-necks.}
\end{center}
\end{figure}

\end{proof}

We will need the following lemma:
\begin{lem}\label{canonical.metrics}
Any connected manifold obtained from  finitely many components endowed with canonical metrics by performing connected sums and attaching handles is isotopic to  a canonical metric.
\end{lem}

\begin{proof}
Suppose $M$ is a connected sum of finitely many components with canonical metrics. Let $S_1, \dots,S_N$ be the corresponding
principal spheres. These spheres are joined to each other by finitely many necks $N_1,\dots,N_k$. We  allow the existence of necks connecting
a principal sphere to itself. We will prove the result by induction on the number of necks $k$.

If there is only one neck, then we have at most  two components. If there is only one component connected to itself by a 
neck, then the result is evidently a canonical metric. If there are exactly two components, 
let  $S_1$ and $S_2$ be their principal spheres. The connected sum $S_1 \# S_2$ can be deformed to a single round sphere, since it
is locally conformally flat.  Therefore the manifold $M$ is isotopic to  a canonical metric. 

Suppose there are $a \geq 2$  necks. Let $N$ be a neck with central 2-sphere $S_N$.  If 
$S_N$ disconnects the manifold, then $M = M_1 \#M_2$ where $M_1$ and $M_2$ are connected sums, with less than $a$ necks each, of finitely many components endowed with canonical metrics.  The induction hypothesis implies that each $M_i$ is isotopic to
 a canonical metric. Therefore our original manifold can be  deformed into 
 a connected sum of two canonical components by just one neck. The result then follows from the $a=1$ case. If $S_N$ does not disconnect the manifold, then $M$ can be obtained as a configuration of $a-1$ necks with one handle attached. It follows  from the induction hypothesis that our manifold is isotopic to a canonical metric with
 one handle attached. Since that is a canonical component by itself, the result is proved.     
\end{proof}

\begin{proof} [Proof of the Main Theorem] Let $g_0$ be  a  positive scalar curvature metric on $M^3$, which can be scaled to be
normalized. Let $(M^3_i,g_i(t))_{t \in [t_i,t_{i+1})}$, $0 \leq i \leq j$, be the Ricci flow with surgery given by Theorem \ref{ricci.flow.surgery}, with initial condition $(M^3,g_0)$. 

\begin{figure}
\begin{center}
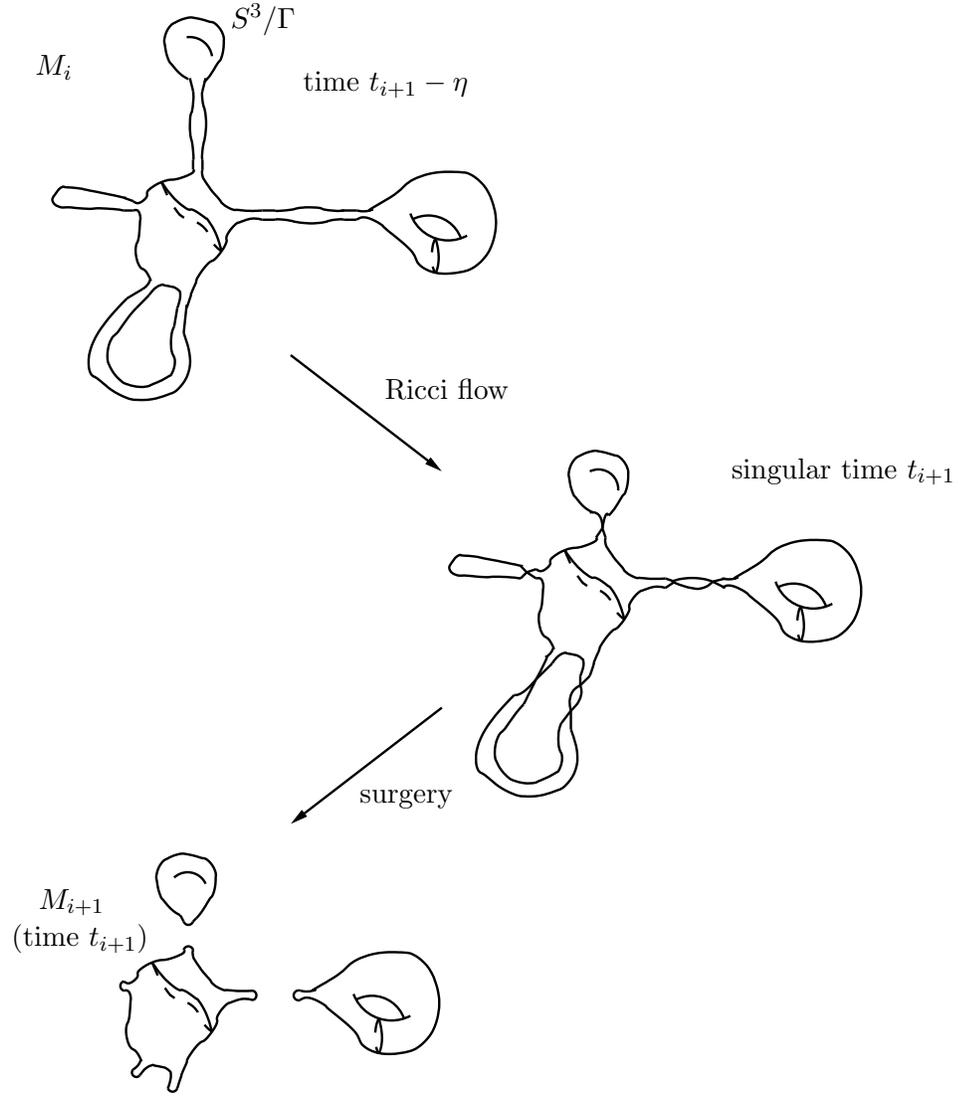
\caption{Surgery at time $t_{i+1}$.}
\end{center}
\end{figure}

Let $\mathcal{A}_i$ be the assertion that the restriction of $g_i(t_i)$ to each component of $M_i$ is isotopic to  a canonical metric.

{\it Claim 1.} If $i<j$ and $\mathcal{A}_{i+1}$ holds, so does $\mathcal{A}_{i}$.

\begin{figure}
\begin{center}
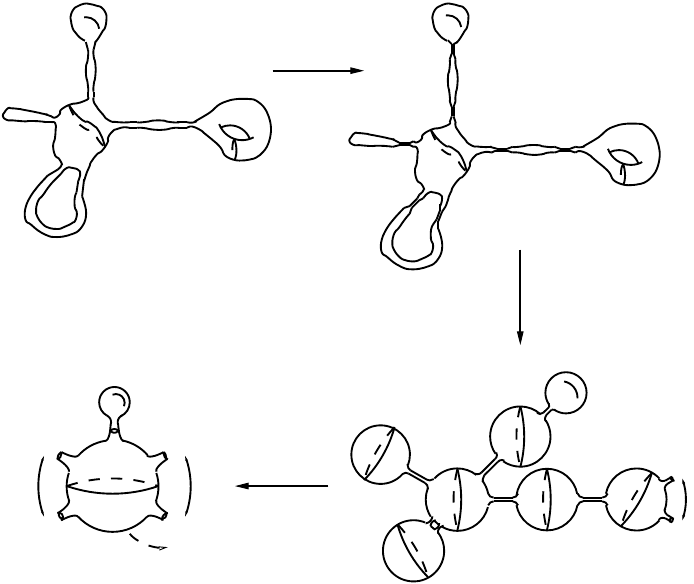
\caption{Isotopy through the singular time $t_{i+1}$.}
\end{center}
\end{figure}

Set $\rho = \delta_i \, r_{i}$ and $h=h(\rho,\delta_i)$, and let $\Omega(t_{i+1})$ and $\Omega_{\rho}(t_{i+1})$ be as in 
Section \ref{section.ricci}. We will denote by  $\Omega^{big}(t_{i+1})$  the union of the finitely many components of $\Omega(t_{i+1})$ that intersect $\Omega_\rho(t_{i+1})$.

The open set $\Omega^{big}(t_{i+1})$ contains a finite collection of disjoint $2\varepsilon$-horns $\mathcal{H}_1, \dots, \mathcal{H}_l$, with boundary contained in $\Omega_{\rho/2C}$, and  such
that the complement of the union of the interiors of these horns is a  compact 3-manifold with boundary that contains $\Omega_\rho$. For each
$1 \leq k \leq l$, let $N_k \subset \mathcal{H}_k$ be the $\delta_i$-neck, centered at $y_k$ with $R_{g_i(t_{i+1})}(y_k)=h^{-2}$, as
explained in Section \ref{section.ricci}. 
Let $S_k$ be the central sphere of this neck, which is oriented so that the positive $s$-direction 
points towards the end of the horn. If $\mathcal{H}_k^+$ is the  complementary component of $S_k$ in $\mathcal{H}_k$ which contains the end of the horn, recall that the continuing
region  $C_{t_{i+1}}$ is defined as 
the complement of $\bigsqcup_{k=1}^l \mathcal{H}_k^+$ in $\Omega^{big}(t_{i+1})$.

Let $N_k^+,N_k^-$ be the positive and negative halves of $N_k$, respectively. If $t' \in (t_i,t_{i+1})$ is sufficiently close to $t_{i+1}$, the metrics $(1-\mu)g_{t'} + \mu g_i(t_{i+1})$, $\mu \in [0,1]$, have positive scalar curvature on $C_{t_{i+1}} \cup \bigsqcup_{k=1}^l N_k^+$ and induce $\delta_i$-neck structures on each $N_k$.
Let  $(\tilde{M}_{t'},\tilde{g}_{t'})$ be the manifold obtained from $(C_{t_{i+1}} \cup \bigsqcup_{k=1}^l N_k^+,g_{t'})$ by surgery along the central spheres of each $N_k$.
Then $\mu \in [0,1] \mapsto ((1-\mu)g_{t'} + \mu g_i(t_{i+1}))_{surg,\delta}$ is a continuous family of positive scalar curvature metrics 
connecting the manifolds $(\tilde{M}_{t'},\tilde{g}_{t'})$ and $(M_{i+1},G_{t_{i+1}})$. Since $G_{t_{i+1}}=g_{i+1}(t_{i+1})$, it follows  from the assertion $\mathcal{A}_{i+1}$ that  each component of $(\tilde{M}_{t'},\tilde{g}_{t'})$ is isotopic to a canonical metric. 

Let $\psi_k:S^2 \times (-1/\delta,1/\delta) \g N_k$ be
the diffeomorphism of the  $\delta_i$-neck structure on $(N_k,g_{t'})$ such that $y_k \in \psi_k(S^2 \times \{0\}) = S_k(t')$, and $\psi_k(S^2 \times (0,1/\delta)) \cap C_{t_{i+1}}= \emptyset$. Since $\Omega_\rho$ is contained in the complement of the union
of the interiors of the horns,  we have $R(x,t') > r_{i}^{-2}$ for every $x \notin C_{t_{i+1}}$ if $t'$ is sufficiently close to $t_{i+1}$. This implies that every $x \notin C_{t_{i+1}}$ has a $(C, \varepsilon)$-canonical neighborhood at time $t'$. 

Let $(P_j,g_{P_j})$, $j=1,\dots,m$, be the components of the compact manifold obtained from 
$$ 
\Big((M_{t'} -C_{t_{i+1}}) \cup \bigsqcup_{k=1}^l N_k^-,g_{t'}\Big)
$$
by surgery along the central spheres of each $N_k$, and replacing the negative halves of the necks by caps. It is clear that
 every point of $(P_j,g_{P_j})$, for each $1\leq j \leq m$, has a $(C,\varepsilon)$-canonical neighborhood, since the surgery caps are caps of type $B$. Because of the presence of
 surgery caps, it follows from 
Proposition \ref{canonical.neighborhood} that the manifolds $(P_j,g_{P_j})$ are diffeomorphic to either $S^3$ or $\re P^3$, and can be deformed to constant curvature manifolds
through positive scalar curvature metrics.

Let $(\hat{M},\hat{g})$ be the compact 3-manifold obtained from $(M_{t'},g_{t'})$ by replacing each region
$$
(\psi_k(S^2 \times (-4,4)),g_{t'})
$$
 with the connected sum
 $$
(\mathcal{S}_k^- \#\, \mathcal{S}_k^+,(g_{t'})_{surg,\delta}^- \# (g_{t'})_{surg,\delta}^+).
$$
The manifold $(\hat{M},\hat{g})$ is a Gromov-Lawson connected sum of the components of  $(\tilde{M}_{t'},\tilde{g}_{t'})$, and
$(P_j,g_{P_j})_{j=1,\dots,m}$. The connected sums are performed at the tips of the surgery caps with sufficiently small
parameters. The parameters are chosen fixed so that the construction applies  to every element in the continuous deformations of the components of $(\tilde{M}_{t'},\tilde{g}_{t'})$ into canonical metrics, and of $(P_j,g_{P_j})_{j=1,\dots,m}$ into constant curvature metrics. 

Therefore it follows from Proposition \ref{connected.sum.continuity} that each component of  $(\hat{M},\hat{g})$ is isotopic to a Gromov-Lawson connected sum of finitely many components endowed with
canonical metrics. It follows from Lemma \ref{canonical.metrics} that  each component of $(\hat{M},\hat{g})$ is isotopic to a canonical metric. 

The manifold $(\hat{M},\hat{g})$ can be continuously deformed back into $(M_{t'},g_{t'})$ through metrics of positive
scalar curvature by Lemma \ref{undoing.surgery}. The Claim follows by using the standard Ricci flow $(g_i(t))_{t\in [t_i,t']}$,
to connect  $(M_{t'},g_{t'})$ and $(M_{t_i},g_{t_i})$.

{\it Claim 2.} $\mathcal{A}_j$ holds.

Since $M_{t_{j+1}}=\emptyset$, there exists $\eta>0$ such that every point of $(M_j,g_j(t))$ has a $(C,\varepsilon)$-canonical neighborhood for all $t \in [t_{j+1}-\eta,t_{j+1})$. Since the standard Ricci flow $(M_j,g_j(t))_{t\in [t_j,t_{j+1}-\eta]}$ is a 
continuous path of positive scalar curvature metrics on $M_j$, the Proposition \ref{canonical.neighborhood} implies that $\mathcal{A}_j$ holds.

It follows from backwards induction on $i$ that $\mathcal{A}_0$ holds. Therefore any metric of positive scalar curvature on $M^3$
can be continuosly deformed, through metrics of positive scalar curvature, into a canonical metric. This proves that the moduli space
$\mathcal{R}_+(M)/{\rm Diff}(M)$ is path-connected.


\end{proof}

\section{Applications to General Relativity}\label{relativity}

In this section we will give some applications of Corollary \ref{sphere} to General Relativity. We will prove the path-connectedness of three different spaces of asymptotically flat metrics on $\re^3$: scalar-flat metrics, nonnegative
scalar curvature metrics, and 
trace-free solutions to the Vacuum Constraint Equations. The result about the moduli space
for other 3-manifolds  (with finitely many Euclidean ends and different inside topology) can be derived by adapting the arguments we will use for $\re^3$ - see the final remark of this section.  We refer the reader to  \cite{BARTNIKISENBERG04} for a nice survey  on
the constraint equations.

Let $0 < \alpha < 1$ be a fixed number.

 The weighted H\"{o}lder space $C^{k,\alpha}_\beta(\re^3)$ is defined  as the set of functions $u \in C^{k,\alpha}_{loc}(\re^3)$ such that the norm
 \begin{eqnarray*}
 \|u\|_{C_\beta^{k,\alpha}} &=& \sum_{i=0}^k \sup_{x \in \re^3} \rho^{i-\beta}(x) |\nabla^i u|(x) \\
 &&+ \sup_{x,y \in \re^3} (\min \rho(x), \rho(y))^{k+\alpha-\beta}\frac{|\nabla^ku(x)-\nabla^ku(y)|}{|x-y|^\alpha}
 \end{eqnarray*}
 is finite, where $\rho(x) = (1+|x|^2)^\frac12$.
 
 It is also convenient to consider the spaces $D^{k,\alpha}_{-3}= C^{k,\alpha}_{-3} \cap L^1$ for $k \geq 0$, and 
 $E^{k,\alpha}_{-1} = \{u \in C^{k,\alpha}_{-1}:\Delta u \in L^1\}$ for $k \geq 2$ (see \cite{SMITHWEINSTEIN04}), with the norms 
 \begin{eqnarray*}
 \|f\|_{D^{k,\alpha}_{-3}} &=& \|f\|_{C^{k,\alpha}_{-3}} + \|f\|_{L^1},\\
 \|v\|_{E^{k,\alpha}_{-1}} &=& \|v\|_{C^{k,\alpha}_{-1}} + \|\Delta v\|_{L^1}.
 \end{eqnarray*}
 Here $\Delta$ denotes the Euclidean Laplacian.
 
 We will need the following result which can be found in \cite{SMITHWEINSTEIN04}:
 \begin{thm}[\cite{SMITHWEINSTEIN04}]\label{fredholm.alternative}
 Let $g$ be a metric on $\re^3$ such that $g-\delta \in C^{k-1,\alpha}_{-\tau}$, for $\tau >0$. Suppose $h \in C^{k-2,\alpha}_{-\nu}$ is 
 a function, $\nu >2$. The operator 
 $$
 \Delta_g - h: E^{k,\alpha}_{-1} \g D^{k-2,\alpha}_{-3}
 $$
 is an isomorphism if and only if it is injective.
 \end{thm}

 Let 
 $$
 \mathcal{M}_{1} = \{{\rm metrics \ } g {\rm \ on \ } \re^3: g_{ij}-\delta_{ij} \in C_{-1}^{2,\alpha} {\rm \ and \ } R_g = 0\}.
 $$

 \begin{prop}\label{deforming.infinity}
 Let $g \in \mathcal{M}_{1}$. Then there exists a $C_{-1}^{2,\alpha}$-continuous path $\mu \in [0,1] \g g_\mu \in \mathcal{M}_1$ such that
 $g_0=g$,  $g_1$ is smooth everywhere and  conformally flat  outside a compact set.
 \end{prop}
 
\begin{proof}
 Let $0 \leq \eta\leq 1$ be a smooth cutoff function such that $\eta(t) = 1$ for $t \leq 1$ and $\eta(t)=0$ for $t \geq 2$. Set
 $\eta_R(t) = \eta(t/R)$ for $R>0$.
 
 Given $g \in \mathcal{M}_1$ and $R>0$, we define $g_R= (1-\eta_R) \delta + \eta_Rg$. We can also approximate $g_R$ by a smooth metric $g_R'$,
 such that $\|g_R-g_R'\|_{C^{2,\alpha}(B_{4R}(0))}$ is small and $g_R'=g_R=\delta$ if $|x| \geq 3R$.
 
 Given $\gamma >0$, it is not difficult to see that for any $\varepsilon>0$, there exists $R_0>0$ such that if $R \geq R_0$ and $\mu \in [0,1]$, we have
 $$
 \|g_{R,\mu} - g\|_{C^{2,\alpha}_{-1+\gamma}} \leq \varepsilon,
 $$
 where $g_{R,\mu} = (1-\mu)\, g + \mu\, g_R'$.
 
 It follows by the Maximum Principle that $\Delta_g:E^{2,\alpha}_{-1} \g D^{0,\alpha}_{-3}$ is injective, thus an isomorphism by Theorem 
 \ref{fredholm.alternative}. We can
 also check that the conformal Laplacian $L_{g_{R,\mu}}=\Delta_{g_{R,\mu}}-\frac18 R_{g_{R,\mu}}$ of $g_{R,\mu}$ is close to $\Delta_g:E^{2,\alpha}_{-1} \g D^{0,\alpha}_{-3}$
 in the operator norm for all $\mu \in [0,1]$, if $R$ is sufficiently large. In that case $L_{g_{R,\mu}}:E^{2,\alpha}_{-1} \g D^{0,\alpha}_{-3}$
 is an isomorphism.
 
 We fix $R$ sufficiently large. It is not difficult to check that if  $\overline{g}_{ij}-\delta_{ij} \in C^{2,\alpha}_{-1}$ and
 $\sum_{i,j} (\d_i\d_j\overline{g}_{ij}-\d_i\d_i\overline{g}_{jj}) \in L^1$, then $R_{\overline{g}} \in L^1$. Since $g \in \mathcal{M}_1$, we have
  $\sum_{i,j} (\d_i\d_jg_{ij}-\d_i\d_ig_{jj})(x)=O(|x|^{-4})$. Since $g_{R}'$ is flat outside a compact set, we conclude that 
  $R_{g_{R,\mu}} \in L^1$ for every $\mu \in [0,1]$.
  
  Let $v_{R,\mu} \in E^{2,\alpha}_{-1}$ be the unique solution to $L_{g_{R,\mu}}(v_{R,\mu})=\frac18 R_{g_{R,\mu}} \in  D^{0,\alpha}_{-3}$, and
 set $u_{R,\mu} =  1 + v_{R,\mu}$.
 
 The Proposition follows if we set  $g_\mu = u_{R,\mu}^4 g_{R,\mu}$.
 \end{proof}

 Let $I(z)=\frac{z}{|z|^2}$ be the inversion with respect to the unit sphere $\d B_1(0) \subset \re^3$. Notice that $I^*(\delta)(z)=|z|^{-4}\delta$.

 \begin{prop}\label{compactification.prop}
 Let $g \in \mathcal{M}_1$ be smooth and such that there exist $R>0$ and  a positive smooth function $u \in 1 + C^{2,\alpha}_{-1}$ with
 $g(z)=u^4(z) \delta$ if $|z| \geq 3R$. Given $p \in S^3$,  there exist a smooth metric $\overline{g}$ on $S^3$ of positive Yamabe quotient, and a diffeomorphism $\varphi:\re^3 \g S^3-\{p\}$ such that:
 \begin{itemize}
 \item [1)] $exp_{\, p, \overline{g}}^{-1}(\varphi(z)) = I(z)$ if $|z| \geq 4R$,
\item[2)] $\varphi_*(g)=G^4\overline{g}$, where $G$ is the Green's function of the conformal Laplacian $L_{\overline{g}}$ with pole at $p$, i.e., the
distributional solution to $L_{\overline{g}}(G)= -\sigma_{2}\, \delta_p$, where $\sigma_2={\rm area}(S_1^2(0))$.
\end{itemize}
 \end{prop}
 Recall that a metric $g$ on a compact manifold is of positive Yamabe quotient if and only if
 $\lambda_1(L_g)>0$. This is equivalent to saying that the conformal class of $g$ contains a metric of positive scalar curvature.
 
 \begin{proof}
  Let $\psi: B_1(0) \subset \re^3 \g S^3$ be a coordinate chart with $\psi(0)=p$, and  let $\varphi:\re^3 \g S^3 -\{p\}$ be
 a diffeomorphism such that $\psi^{-1}(\varphi(z)) = I(z)$.  The maps $\psi$ and $\varphi$ can be chosen as inverses of stereographic
 projections, for instance.
 
 Let $g \in \mathcal{M}_1$ be smooth and such that there exists a positive smooth function $u \in 1 + C^{2,\alpha}_{-1}$ with
 $g(z)=u^4(z) \delta$ if $|z| \geq 3R$. Let $v : \re^3 \g \re$ be a smooth positive function such that $v(z)=|z|\, u(z)$ for $|z|  \geq 3R$. We define a metric $g'$ so that $g= v^4g'$. Hence $g'(z)=|z|^{-4}\delta = I^*(\delta)$ for $|z| \geq 3R$. 
 
 We define $\overline{g}=\varphi_*(g')$, so we can write 
 $\varphi_*(g)=[v \circ \varphi^{-1}]^4 \overline{g}$. Notice that, around $p$, 
 \begin{eqnarray*}
 \overline{g} &=& \varphi_*(|z|^{-4}\delta) \\
 &=& \psi_*(I_*(|z|^{-4} \delta))\\
 &=& \psi_*(\delta).
 \end{eqnarray*}
 Therefore $\psi^*(\overline{g})_{ij}=\delta_{ij}$. Hence $\psi=exp_{\, p, \overline{g}}$, and the assertion (1) follows.
 
 Set $G = v \circ \varphi^{-1}$. Since $\varphi_*(g)$ is scalar-flat,  we have that $L_{\overline{g}}(G)=0$ on $S^3\setminus \{p\}$.
 Since $G>0$ and $\lim_{x \g 0} |x|\, G(\psi(x)) \g 1$, we have that $G$ is a solution to $L_{\overline{g}}(G)=-\sigma_{2}\, \delta_p$ in the distributional sense. The existence of such a function implies that the Yamabe quotient of $(S^3,[\overline{g}])$ is positive. This finishes the proof.
 \end{proof}

 

 \begin{lem}\label{extending.diffeos}
 Let $\mu \in [0,1] \mapsto \overline{g}_\mu$   be a continuous family of smooth Riemannian metrics on $S^3$.  Let $p \in S^3$, and suppose
 that $\{e_i(\mu)\}_{\mu \in [0,1]}$ is a positively oriented $\overline{g}_\mu$-orthonormal basis of $T_pS^3$ depending continuously on $\mu$. Then there exist $M_0>0$ and  a continuous
 family of diffeomorphisms $\varphi_\mu : \re^3 \g S^3-\{p\}$ such that 
 $$
 \varphi_\mu (z) = exp_{\,p,\overline{g}_\mu}\Big(|z|^{-2} \sum_i z_i e_i(\mu)\Big),
 $$
 if $|z| \geq M_0$. 
 \end{lem}
 \begin{proof}
 Let $\psi_\mu = exp_{\,p,\overline{g}_\mu}$. There exists $\beta>0$ such that $\psi_\mu:B_\beta(0)\subset \re^3 \g V_\mu \subset S^3$ is a continuous family of diffeomorphisms with $\psi_\mu(0)=p$ and $(\psi_\mu)_*\,\cdot\d_i=e_i(\mu)$ for all $\mu \in [0,1]$. Hence
 $f_\mu = \psi_\mu \circ \psi_0^{-1} : V_0 \g V_\mu$ is a continuous family of orientation preserving local diffeomorphisms which fix $p \in S^3$. 
 It follows from Theorem 5.5 in \cite{PALAIS59} that an orientation preserving local diffeomorphism which fixes $p$ coincides with an
 ambient diffeomorphism in a sufficiently small neighborhood of $p$. It is not difficult to check, due to the explicit constructions
 of \cite{PALAIS59} and since the interval $[0,1]$ is compact, that we can choose the extensions so that they depend continuosly on the 
 parameter $\mu$. Hence there exist a neighborhood $W_0$ of $p$ and a continuous family of ambient diffeomorphisms $F_\mu:S^3 \g S^3$, $\mu \in [0,1]$, such that $F_\mu = f_\mu$ on $W_0$. 
 
 Let us now suppose that we have found a diffeomorphism $\varphi_0:\re^3 \g S^3 -\{p\}$ such that $\psi_0^{-1} \circ \varphi_0 = I$ outside
 a compact set. Then we can define the diffeomorphism $\varphi_\mu = F_\mu \circ \varphi_0:\re^3 \g S^3 -\{p\}$, $\mu \in [0,1]$ and it follows
 that $\psi_\mu^{-1} \circ \varphi_\mu = I$ outside a fixed compact set. It remains to find $\varphi_0$.
 
 In order to do that we choose stereographic projections $\varphi$  and $\psi$ as in the proof of Proposition \ref{compactification.prop} so that
 $\psi^{-1} \circ \varphi = I$. Then we choose $h$ to be an ambient diffeomorphism of $S^3$  extending $\psi_0 \circ \psi^{-1}$, and  set
 $\varphi_0 = h \circ \varphi$. It follows immediately that $\varphi_0:\re^3 \g S^3 -\{p\}$ is a diffeomorphism and $\psi_0^{-1} \circ \varphi_0 = I$ outside a compact set. This finishes the proof of the lemma. 
 \end{proof}
 
 \begin{thm}\label{connectedness.1}
  The set $\mathcal{M}_{1}$ is path-connected in the $C^{2,\alpha}_{-1}$ topology.
 \end{thm}

 \begin{proof}
 It suffices to prove that any two metrics satisfying the assumptions of Proposition \ref{compactification.prop} are in
 the same path-connected component of $\mathcal{M}_{1}$.
 
 Let $g^{(0)},g^{(1)} \in \mathcal{M}_1$ be smooth metrics which are conformally flat outside a compact set. Given $p \in S^3$, 
 it
 follows from Proposition \ref{compactification.prop} that there exist  smooth metrics $\overline{g}^{(0)},\overline{g}^{(1)}$ on $S^3$ of positive Yamabe quotient, and  diffeomorphisms $\varphi^{(0)},\varphi^{(1)}:\re^3 \g S^3-\{p\}$ such that:
 \begin{itemize}
 \item [1)] $exp_{\, p, \overline{g}^{(i)}}^{-1}(\varphi^{(i)}(z)) = I(z)$ if $|z| \geq 4R$, $i=0,1$,
\item[2)] $\varphi^{(i)}_*(g^{(i)})=G_i^4\overline{g}^{(i)}$, where $G_i$ is the Green's function of the conformal Laplacian $L_{\overline{g}^{(i)}}$ with pole at $p$, $i=0,1$.
\end{itemize}

Since the conformal class of a metric of positive Yamabe quotient contains a positive scalar curvature metric, it follows from
 Corollary  \ref{sphere} that there exists a continuous family $(\overline{g}_\mu)_{\mu \in [0,1]}$ of $C^\infty$ Riemannian metrics
 on $S^3$  of positive Yamabe quotient such that $\overline{g}_0=\overline{g}^{(0)}$ and $\overline{g}_1 = \overline{g}^{(1)}$. Let
 $\varphi_\mu : \re^3 \g S^3-\{p\}$ be the family of diffeomorphisms given by  Lemma \ref{extending.diffeos}. If $G_\mu$ denotes the
 Green's function of $L_{\overline{g}_\mu}$ with pole at $p$, it follows then from standard arguments in elliptic linear theory and from the expansion of $G_\mu$ in
 inverted normal coordinates (see \cite{LEEPARKER87}) that the family $\varphi_\mu^*(G_\mu^4\, \overline{g}_\mu)$ is continuous in $\mathcal{M}_1$.
 
 It remains to prove that $\varphi_i^*(G_i^4\, \overline{g}^{(i)})$ and $g^{(i)}=(\varphi^{(i)})^*(G_i^4\, \overline{g}^{(i)})$ are in the same path-connected component of $\mathcal{M}_1$,
 $i=0,1$. In order to see that notice that, for each $i=0,1$, $(\varphi^{(i)})^{-1} \circ \varphi_i : \re^3 \g \re^3$ is a diffeomorphism which
 coincides with the identity outside a compact set. Then there exists a continuous family of diffeomorphisms $F_{\mu,i}:\re^3 \g \re^3$ (see \cite{CERF68})
 such that $F_{0,i} = (\varphi^{(i)})^{-1} \circ \varphi_i$,  $F_{1,i} = id$, and $F_{\mu,i}(z)=z$ for all $\mu \in [0,1]$ and $|z| \geq R_0$.
 Hence $F_{\mu,i}^*(g^{(i)})$ is a continuous path in $\mathcal{M}_1$ joining $\varphi_i^*(G_i^4\, \overline{g}^{(i)})$ and $g^{(i)}$. This
 finishes the proof of the theorem.
 \end{proof}

 Let 
 $$
 \mathcal{M}_{2} = \{{\rm metrics \ } g {\rm \ on \ } \re^3: g_{ij}-\delta_{ij} \in C_{-1}^{2,\alpha}, R_g \in L^1 {\rm \ and \ } R_g \geq 0\}.
 $$
 
\begin{thm}\label{connectedness.2}
  The set $\mathcal{M}_{2}$ is path-connected in the $C^{2,\alpha}_{-1}$ topology.
 \end{thm}
 
 \begin{proof}
 Let $g \in \mathcal{M}_2$. Since $R_g \geq 0$, it follows by the Maximum Principle and Theorem \ref{fredholm.alternative} that the operator 
 $$
 L_g=\Delta_g - \frac18 R_g: E^{2,\alpha}_{-1} \g D^{0,\alpha}_{-3}
 $$
 is an isomorphism. Since $R_g \in L^1$, we define $v\in E^{2,\alpha}_{-1}$ to be the solution of $L_gv=\frac18 R_g$, and let 
 $u=1+v$. Hence $L_gu =0$, and $u>0$ by the Maximum Principle.
 
 Define $g_\mu = u_\mu^4g$,   where $\mu \in [0,1]$ and $u_\mu = (1-\mu) + \mu\, u$.
 
 The result follows from Theorem \ref{connectedness.1}, since $g_1 \in \mathcal{M}_1$ and $\mathcal{M}_1 \subset \mathcal{M}_2$. 
 \end{proof}

 We say that $(g,h)$ is an {\it asymptotically flat initial data set} on $\re^3$ if $g$ is a Riemannian metric on $\re^3$ such that
 $g_{ij}-\delta_{ij} \in C^{2,\alpha}_{-1}$, and $h$ is a symmetric $(0,2)$-tensor with $h_{ij} \in C^{1,\alpha}_{-2}$. The {\it Vacuum Constraint Equations} on $\re^3$ are
 \begin{itemize}
 \item [a)] $R_g +(tr_g\, h)^2 - |h|_g^2=0$,
 \item [b)] and $\nabla _i h^{i}_{\ j}-\nabla_j(tr_g\, h) = 0$.
 \end{itemize}
 
 Let $\mathcal{M}_3$ be the set of all asymptotically flat initial data sets $(g,h)$ on $\re^3$ which satisfy the
 Vacuum Constraint Equations, and such that $tr_g\, h=0$.

We  have
 \begin{thm}\label{connectedness.3}
 The set $\mathcal{M}_3$ is path-connected in the $C^{2,\alpha}_{-1} \times C^{1,\alpha}_{-2}$-topology.
 \end{thm}

 \begin{proof} Let $\tilde{\mathcal{M}}_3$ be the set of all asymptotically flat initial data sets on $\re^3$ such that
 \begin{itemize}
 \item [a)] $tr_g\, h = 0$,
 \item [b)] $R_g \geq |h|_g^2$ and $R_g \in L^1$,
 \item [c)] and $(div_g\, h)_j:=\nabla _i h^{i}_{\ j}= 0$.
 \end{itemize}
 
 Given $(g,h) \in \tilde{\mathcal{M}}_3$,  the path $\mu \in [0,1] \mapsto (g, (1-\mu)h)$ is continuous and contained
 in $\tilde{\mathcal{M}}_3$. It follows then from Theorem \ref{connectedness.2} that $\tilde{\mathcal{M}}_3$ is path-connected.
 It suffices to show that there exists a continuous and surjective map $F: \tilde{\mathcal{M}}_3 \g \mathcal{M}_3$. This
 can be accomplished by the conformal method as follows.
 
 If $(g,h) \in \tilde{\mathcal{M}}_3$, we want to find  a positive function $u$ such that $(\hat{g},\hat{h}) \in \mathcal{M}_3$,
 where $\hat{g}=u^4\, g$, and $\hat{h}=u^{-2}\, h$. This is equivalent to solving $R_{\hat{g}}=|\hat{h}|_{\hat{g}}^2$, since
 it is easily checked that $tr_{\hat{g}}\, \hat{h} = 0$ and $div_{\hat{g}}\, \hat{h}=0$ for any $u>0$.
 
 The equation $R_{\hat{g}}=|\hat{h}|_{\hat{g}}^2$ translates into the Lichnerowicz equation
 \begin{equation}\label{lichnerowicz}
 \Delta_g u - \frac18 R_g u + \frac18 |h|_g^2 u^{-7} = 0.
 \end{equation}

 Since $R_g \geq |h|_g^2$, the function $u_+=1$ is a supersolution to the equation (\ref{lichnerowicz}). 
 
 We can also solve
 $L_gv=\frac18 R_g$ as in the proof of  Theorem \ref{connectedness.2}, and set $u_-=1+v$. It follows by the Maximum Principle
 that $0<u_- \leq 1$. We also have
 \begin{eqnarray*}
  \Delta_g u_- - \frac18 R_g u_- + \frac18 |h|_g^2 u_-^{-7} = \frac18 |h|_g^2 u_-^{-7} \geq 0.
 \end{eqnarray*}
 Therefore $u_-$ is a subsolution to the equation (\ref{lichnerowicz}), with $u_- \leq u_+$.
 
 The method of sub and supersolutions gives the existence of a positive function $u \in 1+ D^{2,\alpha}_{-1}$ which
 solves the equation (\ref{lichnerowicz}) and such that $u_- \leq u \leq u_+$. If $u_1$ and $u_2$ are solutions, we have
 $$
 \Delta_g(u_1-u_2) = \frac18 R_g (u_1-u_2) + \frac18 |h|_g^2 (u_2^{-7}-u_1^{-7}). 
 $$
 It follows by the Maximum Principle that $u_1=u_2$.
 
 Therefore the map $F: \tilde{\mathcal{M}}_3 \g \mathcal{M}_3$ given by
 $$
 F\, (g,h) = (u^4\, g, u^{-2}\, h)
 $$
 is well-defined. It is also surjective since it restricts to the identity on $\mathcal{M}_3$. The continuity
 of $F$ follows from standard elliptic regularity on weighted spaces. This finishes the proof of the theorem.
 \end{proof}

 {\it Remark:}  Similar results can be derived for the corresponding moduli spaces of other 3-manifolds. Suppose, for instance, that $(M^3,g)$ is an asymptotically flat manifold with zero scalar curvature and $n$ ends. The arguments of this section
 can be adapted to prove that this manifold can be deformed, through asymptotically flat metrics of zero scalar curvature, into a manifold isometric to the blow-up of
 a canonical metric (as in Section \ref{proofs}) at $n$ distinct points. The blow-up of a compact manifold $(\overline{M}^3,\overline{g})$ of positive
 scalar curvature at the points $x_1,\dots,x_n \in M$ is the scalar-flat and asymptotically flat manifold  $\hat{g} = (\sum_{i=1}^n G_{x_i})^4 \overline{g}$, where $G_{x_i}$ is the Green's function of the conformal
 Laplacian $L_{\overline{g}}$ with pole at $x_i$. The idea is again to first deform $(M^3,g)$ into a manifold which can be conformally compactified (as in Propositions \ref{deforming.infinity} and \ref{compactification.prop}), and then apply
 the Main Theorem of this paper.
 


\footnotesize

{\sc \noindent Fernando Coda Marques, Instituto de Matem\'{a}tica Pura e Aplicada (IMPA), Estrada Dona Castorina 110, 22460-320, Rio de Janeiro - RJ,
Brazil}

\noindent {\it Email address : \ } coda@impa.br


\begin{thebibliography}{10}

\bibitem{BARTNIKISENBERG04}
R.~Bartnik and J.~Isenberg.
\newblock The constraint equations.
\newblock {\em The Einstein equations and the large scale behavior of gravitational fields}, 1--38, Birkhäuser, Basel, 2004

\bibitem{BESSONetal08}
L.~Bessières, G.~Besson, M.~Boileau, S.~Maillot and J.~Porti.
\newblock Collapsing irreducible  3-manifolds with nontrivial fundamental group.
\newblock {\em Invent. Math.}, 179(2): 435--460,  2010

\bibitem{BESSONetal09}
L.~Bessières, G.~Besson, M.~Boileau, S.~Maillot and J.~Porti.
\newblock Geometrisation of 3-manifolds.
\newblock {\em EMS Tracts in Mathematics}, 13, European Mathematical Society (EMS), Zurich,  2010

\bibitem{BOTVINNIKGILKEY96}
B.~Botvinnik and P.~Gilkey.
\newblock Metrics of positive scalar curvature on spherical space forms.
\newblock {\em Canad. J. Math.}, 48(1):64--80, 1996


\bibitem{CAOZHU06}
H-D.~Cao and X-P.~Zhu.
\newblock A complete proof of {P}oincaré and {G}eometrization conjectures - {A}pplication of the {H}amilton-{P}erelman theory of the {R}icci flow.
\newblock {\em Asian J. Math.}, 10(2):169--492, 2006

\bibitem{CARR88}
R.~Carr.
\newblock Construction of manifolds of positive scalar curvature.
\newblock {\em Trans. Amer. Math. Soc.}, 307:63--74, 1988

\bibitem{CERF68}
J.~Cerf.
\newblock Sur les difféomorphismes de la sphère de dimension trois $(\Gamma \sb{4}=0)$.
\newblock {\em Lecture Notes of Math.}, 53,  Springer-Verlag, Berlin-New York, 1968

\bibitem{CHENLUTIAN06}
X.~Chen, P.~Lu and G. Tian.
\newblock A note on uniformization of Riemann surfaces by Ricci flow.
\newblock {\em Proc. Amer. Math. Soc.}, 134:3391--3393, 2006



\bibitem{CHOW91}
B.~Chow.
\newblock The Ricci flow on the 2-sphere.
\newblock {\em J. Differential Geom.}, 33:325--334, 1991


\bibitem{COLDINGMINICOZZI05}
T.~Colding and B. Minicozzi.
\newblock Estimates for the extinction time for the Ricci flow on certain 3-manifolds and a question of Perelman.
\newblock {\em J. Amer. Math. Soc.}, 18(3):561--569, 2005


\bibitem{DERHAM50}
G.~de Rham.
\newblock Complexes à automorphismes et homéomorphie différentiable
\newblock {\em Ann. Inst. Fourier Grenoble}, 2:51--67, 1950

\bibitem{GROMOVLAWSON80}
M.~Gromov and H. Blaine~Lawson,Jr.
\newblock The classification of simply connected manifolds of positive scalar curvature.
\newblock {\em Ann. of Math.}, 111:423--434, 1980


\bibitem{GROMOVLAWSON83}
M.~Gromov and H. Blaine~Lawson,Jr.
\newblock Positive scalar curvature and the Dirac operator on complete Riemannian manifolds. 
\newblock {\em Inst. Hautes Études Sci. Publ. Math.}, 58:83--196, 1983


\bibitem{HAMILTON82}
R.~Hamilton.
\newblock Three-manifolds with positive Ricci curvature.
\newblock {\em J. Differential Geom.}, 17(2):255--306, 1982


\bibitem{HAMILTON88}
R.~Hamilton.
\newblock The Ricci flow on surfaces. In: {M}athematics and {G}eneral {R}elativity. (Santa Cruz, CA, 1986)
\newblock {\em Contemp. Math.}, 71:237--262, Amer. Math. Soc., 1988


\bibitem{HAMILTON97}
R.~Hamilton.
\newblock Four-manifolds with positive isotropic curvature. 
\newblock {\em Comm. Anal. Geom.}, 5(1):1--92,  1997

\bibitem{HATCHER83}
A.~Hatcher.
\newblock A proof of the Smale conjecture, ${\rm Diff}(S\sp{3})\simeq {\rm O}(4)$.
\newblock {\em Ann. of Math.}, 117(3):553--607,  1983


\bibitem{HEMPEL04}
J.~Hempel.
\newblock 3-Manifolds.
\newblock {\em AMS Chelsea Publishing}, American Mathematical Society,  2004


\bibitem{HITCHIN74}
N.~Hitchin.
\newblock Harmonic spinors.
\newblock {\em Adv.  Math.}, 14:1--55, 1974

\bibitem{KLEINERLOTT08}
B.~Kleiner and J. Lott.
\newblock Notes on {P}erelman's papers.
\newblock {\em Geom. Topol.}, 12:2587--2855, 2008

\bibitem{KRECKSTOLZ93}
M.~Kreck and S. Stolz.
\newblock Nonconnected moduli spaces of positive sectional curvature metrics.
\newblock {\em J. Amer. Math. Soc.}, 6(4):825--850, 1993

\bibitem{KUIPER49}
N.~Kuiper.
\newblock On conformally-flat spaces in the large. 
\newblock {\em Ann. of Math.}, 50(2):916--924, 1949.



\bibitem{KUIPER50}
N.~Kuiper.
\newblock On compact conformally Euclidean spaces of dimension $>2$.
\newblock {\em Ann. of Math.}, 52(2):478--490, 1950.




\bibitem{LEEPARKER87}
J.~Lee and T.~Parker.
\newblock The {Y}amabe problem.
\newblock {\em Bull. Amer. Math. Soc.}, 17:37--91, 1987

\bibitem{MILNOR62}
J.~Milnor.
\newblock A unique factorization theorem for 3-manifolds.
\newblock {\em Amer. J. Math.}, 84:1--7, 1962

\bibitem{MORGANTIAN07}
J.~Morgan and G.~Tian.
\newblock Ricci flow and the {P}oincar\'{e} conjecture.
\newblock {\em Clay Mathematics Monographs}, American Mathematical Society, 2007

\bibitem{MORGANTIAN08}
J.~Morgan and G.~Tian.
\newblock Completion of the proof of the geometrization conjecture.
\newblock {\em arXiv:0809.4040v1 [math.DG]}, 2008

\bibitem{MUNKRES60}
J.~Munkres.
\newblock Differentiable isotopies on the $2$-sphere.
\newblock {\em Michigan Math. J.}, 7(3):193--197, 1960


\bibitem{NIRENBERG53}
L.~Nirenberg.
\newblock The {W}eyl and {M}inkowski problems in differential geometry in the large.
\newblock {\em Comm. Pure Appl. Math.}, 6:337--394, 1953

\bibitem{PALAIS59}
R.~Palais.
\newblock Natural operations on differential forms.
\newblock {\em Trans. Amer. Math. Soc.}, 92:125--141, 1959





\bibitem{PERELMAN02}
G.~Perelman.
\newblock The entropy formula for the {R}icci flow and its geometric applications.
\newblock {\em arXiv:math/0211159v1 [math.DG]}, 2002

\bibitem{PERELMAN03A}
G.~Perelman.
\newblock Ricci flow with surgery on three-manifolds.
\newblock {\em arXiv:math/0303109v1 [math.DG]}, 2003
 

\bibitem{PERELMAN03B}
G.~Perelman.
\newblock Finite extinction time for the solutions to the {R}icci flow on certain three-manifolds.
\newblock {\em arXiv:math/0307245v1 [math.DG]}, 2003


\bibitem{ROSENBERG07}
J.~Rosenberg.
\newblock Manifolds of positive scalar curvature: a progress report.
\newblock {\em Surv. Differ. Geom.}, Int. Press, Somerville, MA, 11:259--294, 2007


\bibitem{ROSENBERGSTOLZ01}
J.~Rosenberg and S. Stolz.
\newblock Metrics of positive scalar curvature and connections with surgery. Surveys on surgery theory.
\newblock {\em Ann. of Math. Stud.}, 149:353--386, Princeton Univ. Press, Princeton, NJ, 2001 

\bibitem{SCHOENYAU79}
R.~Schoen and S.T.~Yau.
\newblock On the structure of manifolds of positive scalar curvature.
\newblock {\em Manuscripta Math.}, 28:159--183, 1979

\bibitem{SCHOENYAU79A}
R.~Schoen and S.T.~Yau.
\newblock Existence of incompressible minimal surfaces and the topology of three-dimensional manifolds with nonnegative scalar curvature.
\newblock {\em Ann. of Math.}, 110(1):127--142, 1979


\bibitem{SCHOENYAU88}
R.~Schoen and S.-T.~Yau.
\newblock Conformally flat
manifolds, Kleinian groups, and scalar curvature.
\newblock {\em Invent. Math.}, 92:47--71, 1988.



\bibitem{SCHOENYAU94}
R.~Schoen and S.-T.~Yau.
\newblock Lectures on
Differential Geometry.
\newblock {\em Conference Proceedings and Lecture Notes
in Geometry and Topology, International Press Inc.}, 1994.

\bibitem{SMALE59}
S.~Smale.
\newblock Diffeomorphisms of the $2$-sphere.
\newblock {\em Proc. Amer. Math. Soc.}, 10:621--626, 1959
 
\bibitem{SMITHWEINSTEIN04}
B.~Smith and G.~Weinstein.
\newblock Quasiconvex foliations and asymptotically flat metrics of non-negative scalar curvature.
\newblock {\em Comm. Anal. Geom.}, 12(3):511--551, 2004

\bibitem{TANNO73}
S.~Tanno.
\newblock Compact conformally flat Riemannian manifolds.
\newblock {\em J. Differential Geometry}, 8:71--74, 1973

 

\bibitem{WEYL16}
H.~Weyl.
\newblock Über die Bestimmung einer geschlossenen konvexen Fläche durch ihr Linienelement.
\newblock {\em Vierteljahrsschr. Naturforsch. Ges. Zür.}, 61:70--72, 1916

\end{thebibliography}
\end{document}